\documentclass[reqno,10pt,dvips]{article}
\usepackage{amsmath,amsfonts,amssymb,amsthm}
\usepackage{enumerate,graphicx,ifthen,bm,geometry,fullpage,fancyhdr,multicol,authblk}
\usepackage{thmtools,thm-restate}
\usepackage{enumitem}
\usepackage[bookmarks=true,implicit=true,pagebackref=false]{hyperref}
\usepackage[all]{hypcap}
\usepackage[small,bf,hang]{caption}
\usepackage[compact]{titlesec}

\setitemize{noitemsep,topsep=0pt,parsep=0pt,partopsep=0pt,parsep=0pt}
\setenumerate{noitemsep,topsep=0pt,parsep=0pt,partopsep=0pt,parsep=0pt}

\titlespacing{\section}{0pt}{2ex}{1ex}
\titlespacing{\subsection}{0pt}{1ex}{0ex}
\titlespacing{\subsubsection}{0pt}{0.5ex}{0ex}
\titlespacing{\paragraph}{0pt}{0.5ex}{1.5ex}

\newcounter{egnum}
\declaretheorem[numberwithin=section]{theorem}
\declaretheorem[sibling=theorem]{proposition}
\declaretheorem[sibling=theorem]{corollary}
\declaretheorem[sibling=theorem,style=definition]{remark}

\numberwithin{equation}{section}
\numberwithin{figure}{section}
\numberwithin{table}{section}

\renewenvironment{proof}[1]{\medskip\noindent\textbf{Proof}\ifthenelse{\equal{#1}{}}{}{~\textbf{of~{#1}:}}}{}

\newcommand{\figwid}{150pt}

\newcommand{\fracslash}[2]{{#1}/{#2}}
\newcommand{\sfrac}[2]{\textstyle\frac{#1}{#2}\displaystyle}
\newcommand{\half}{\frac{1}{2}}
\newcommand{\shalf}{\sfrac{1}{2}}
\newcommand{\sthird}{\sfrac{1}{3}}
\newcommand{\squarter}{\sfrac{1}{4}}
\newcommand{\ds}[1]{\displaystyle{#1}}
\newcommand{\dsm}[1]{$\ds{#1}$}
\newcommand{\st}{\,:\,}
\newcommand{\setst}[2]{\setb{\,#1\st#2\,}}

\newcommand{\bmf}[2]{\bm{#1}\!\rb{#2}}
\newcommand{\mymbox}[1]{\quad\mbox{#1}\quad}
\newcommand{\Cr}[2]{\ifthenelse{\equal{#2}{}}{\cC^{#1}}{\cC^{#1}\!\rb{#2}}}

\newcommand{\pderivslash}[2]{\fracslash{\partial#1}{\partial#2}}
\newcommand{\sgn}[1]{\textsf{sgn}\!\rb{#1}}
\newcommand{\abs}[1]{\left|#1\right|}
\renewcommand{\sgn}[1]{\textsf{sgn}\!\rb{#1}}
\newcommand{\norm}[1]{\left\|#1\right\|}
\newcommand{\specrad}[1]{\scf{\rho}{#1}}
\newcommand{\spectrum}[1]{\scf{\sigma}{#1}}
\newcommand{\detb}[1]{\textsf{det}\!\rb{#1}}
\newcommand{\diag}[1]{\textsf{diag}\!\rb{#1}}
\newcommand{\scf}[2]{#1\!\rb{#2}}
\newcommand{\scif}[3]{#1_{#2}\!\rb{#3}}
\newcommand{\sci}[2]{#1_{#2}}
\newcommand{\scm}[2]{#1^{\rb{#2}}}
\newcommand{\scim}[3]{#1_{#2}^{\rb{#3}}}
\newcommand{\scimf}[4]{#1_{#2}^{\rb{#3}}\!\rb{#4}}
\renewcommand{\vec}[1]{\textbf{#1}}

\newcommand{\vecf}[2]{\vec{#1}\!\rb{#2}}
\newcommand{\vecmf}[3]{\vec{#1}^{\rb{#2}}\!\rb{#3}}
\newcommand{\mat}[1]{\textbf{#1}}
\newcommand{\vecxi}{{\bm{\xi}}}
\newcommand{\veczeta}{{\bm{\zeta}}}
\newcommand{\matPhi}{{\bm{\Phi}}}
\newcommand{\matPhif}[1]{\bmf{\Phi}{#1}}
\newcommand{\bmat}[1]{\overline{\mat{#1}}}

\newcommand{\mati}[2]{\mat{#1}_{#2}}
\newcommand{\matm}[2]{\mat{#1}^{\rb{#2}}}

\newcommand{\rb}[1]{\left(#1\right)}
\newcommand{\sqb}[1]{\left[#1\right]}
\newcommand{\setb}[1]{\left\{#1\right\}}
\newcommand{\sqrb}[1]{\left[#1\right)}
\newcommand{\Nd}[1]{\mathbb{N}^{#1}}
\newcommand{\Cm}{\Cd{m}}
\newcommand{\Cmnz}{\Cdnz{m}}
\newcommand{\Cd}[1]{\mathbb{C}^{#1}}
\newcommand{\Cdnz}[1]{{\mathbb{C}}_{\ne 0}^{#1}}

\newcommand{\Rm}{\Rd{m}}
\newcommand{\Rd}[1]{\mathbb{R}^{#1}}
\newcommand{\Rnp}{\Rdp{n}}
\newcommand{\Rdp}[1]{{\mathbb{R}}_{>0}^{#1}}
\newcommand{\Rnnn}{\Rdnn{n}}
\newcommand{\Rmnn}{\Rdnn{m}}
\newcommand{\Rdnn}[1]{{\mathbb{R}}_{\geq 0}^{#1}}
\newcommand{\matCmm}{\matCd{m}{m}}
\newcommand{\matCd}[2]{\Cd{#1 \times #2}}
\newcommand{\matRnn}{\matRd{n}{n}}
\newcommand{\matRmm}{\matRd{m}{m}}
\newcommand{\matRd}[2]{\Rd{#1 \times #2}}
\newcommand{\matRmmnn}{\Rdnn{m \times m}}
\newcommand{\matRnnnn}{\Rdnn{n \times n}}
\newcommand{\matRdnn}[2]{\Rdnn{#1 \times #2}}
\newcommand{\cC}{\mathcal{C}}
\newcommand{\cE}{\mathcal{E}}
\newcommand{\cK}{\mathcal{K}}
\newcommand{\cR}{\mathcal{R}}

\newcommand{\cX}{\mathcal{X}}
\newcommand{\cY}{\mathcal{Y}}
\newcommand{\hcR}{\widehat{\cR}}

\newcommand{\bB}{\overline{B}}
\newcommand{\bD}{\overline{D}}
\newcommand{\bN}{\overline{N}}

\newcommand{\bW}{\overline{W}}
\newcommand{\bX}{\overline{X}}
\newcommand{\twobytwomatrix}[4]{\begin{pmatrix}#1&#2\\#3&#4\end{pmatrix}}
\newcommand{\threebythreematrix}[9]{\begin{pmatrix}#1&#2&#3\\#4&#5&#6\\#7&#8&#9\end{pmatrix}}

\renewcommand{\lim}{\operatornamewithlimits{\textsf{lim}}}
\renewcommand{\max}[2]{\underset{#1}{\textsf{max}}\setb{#2}}
\newcommand{\maxst}[2]{\textsf{max}\setb{\,#1\st#2\,}}
\renewcommand{\min}[2]{\underset{#1}{\textsf{min}}\setb{#2}}
\newcommand{\minst}[2]{\textsf{min}\setb{\,#1\st#2\,}}

\begin{document}

\pagestyle{fancy}

\renewcommand{\thefootnote}{\fnsymbol{footnote}}

\title{\textbf{Amplification of the Net Reproductive Number by Dispersion for a Matrix Population Model Applicable to the Invasive Round Goby Fish}}

\author[1*]{Matt S. Calder}
\author[2]{Yingming Zhao}
\author[3]{Xingfu Zou}

\affil[1]{Department of Applied Mathematics, University of Western Ontario, \newline 1151 Richmond Street, London, Ontario, Canada, N6A 5B7}
\affil[2]{Aquatic Research and Development Section, Ontario Ministry of Natural Resources, \newline 320 Milo Road, Wheatley, Ontario, Canada, N0P 2P0}
\affil[3]{Department of Applied Mathematics, University of Western Ontario, \newline 1151 Richmond Street, London, Ontario, Canada, N6A 5B7}

\footnotetext[1]{Corresponding author}

\renewcommand{\thefootnote}{\arabic{footnote}}

\maketitle

\begin{abstract}
Matrix population models have proved popular and useful for studying stage-structured populations in quantitative ecology. The goal of this paper is to develop and analyze a general matrix population model that is applicable to the population dynamics of the invasive round goby fish that incorporates both stage structure---larvae, juveniles, and adults---and dispersion. Specifically, we address the issue of whether or not dispersion can amplify the net reproductive number. To this end, we will first review the mathematics of matrix population models with dispersion, particularly those with an Usher demography matrix. Techniques for computing the net reproductive number, like the graph reduction method of de-Camino-Beck and Lewis, will be discussed. A common theme will be the usefulness of submatrices of relevant matrices obtained by the expunging of rows and columns corresponding to non-newborns. Finally, examples will be provided, including the calculation of the net reproductive number for multiple regions using the graph reduction method of de-Camino-Beck and Lewis, examples where dispersion results in a total population flourishing when the populations would otherwise go extinct, and an application to the invasive round goby fish.
\end{abstract}



\section{Introduction} \label{sec01}

To model the dynamics of a population with distinct stages, within which the vital rates vary little, that influence each other and the times of interest are periodic, such as months or years, the tool of choice is a matrix population model. For models ignoring spatial variation, the other options include ordinary differential equations (continuous time, discrete stages), partial differential equations (continuous time and stages), and integro-difference equations models (discrete time, continuous stages).

\paragraph{Brief history.}

Concepts from what is now referred to as matrix population analysis date back to the end of the Nineteenth Century. However, it was the work of Leslie \cite{Leslie1945,Leslie1948} in the middle of the 1940s that brought the techniques to prominence. Leslie studied age-structured animal populations, but Goodman \cite{Goodman1967} and Keyfitz \cite{Keyfitz1967} applied matrix population analysis to human demographics and demonstrated their equivalence to integro-difference equation models. The famous McKendrick-von Foerster partial differential equation can be viewed as the continuous-time and continuous-stage version of the Leslie population model. See, for example, \S8.1 of \cite{Caswell2001}. Usher \cite{Usher1966,Usher1969a,Usher1969b} generalized the Leslie model, with forestry being the original application, and allowed for more general stage-structured populations.

\paragraph{Standard references.}

The standard reference for matrix population models is Caswell's \cite{Caswell2001}. To quote this oft-quoted book (specifically, Page~xviii of the Preface), ``Matrix population models---carefully constructed, correctly analyzed, and properly interpreted---provide a theoretical basis for population management.'' Further, ``The methods involved may appear daunting ... but population managers deserve the sharpest analytical tools available. Their work is too important to settle for less.''

Standard references for the linear algebra content include \cite{Bellman1970,Gantmacher1959,HornJohnson2013}. The Cushing-Zhou Theorem, which relates the growth rate of a population to the average reproductive output of an individual, can be found in \cite{Cushing2011,CushingZhou1994}. More information and generalizations can be found in \cite{LiSchneider2002}. The graph reduction method which we will use, due to de-Camino-Beck and Lewis, can be found in \cite{de-Camino-BeckLewis2007}. For a by-no-means-exhaustive collection of references using matrix models and/or net reproductive numbers, many of which involve dispersion, we suggest \cite{AllenvandenDriessche2008,CaoZhou2012,Caswell1982a,Caswell1982b,Cushing1998,Cushing2009,Cushing2011,KrkosekLewis2010,LevinGoodyear1980,Velez-EspinoKoopsBalshine2010} and references therein.

\paragraph{Scalar model.}

The simplest (single-stage) linear, discrete-time population model is ${\scf{x}{t+1}=r\,\scf{x}{t}}$, with initial condition ${\scf{x}{0}=x_0}$, where $\scf{x}{t}$ is the number of individuals (usually restricted to females) of a given species at discrete time ${t\in\Nd{}_0}$, with ${\Nd{}_0:=\setb{0,1,2,\ldots}}$, and ${r,x_0 \geq 0}$. Trivially, the solution is ${\scf{x}{t}=r^t\,x_0}$ for ${t\in\Nd{}_0}$, and so $r$ is immediately recognizable as the \emph{growth rate}. If ${r<1}$, then the population decays; if ${r>1}$, then the population grows. Commonly, we write ${r:=f+s}$, where ${f \geq 0}$ quantifies \emph{fecundity} (average reproductive output) and ${s\in\sqrb{0,1}}$ quantifies the \emph{survival probability}. Since ${f\,s^t}$ is the expected reproductive output from a single newborn individual between times $t$ and ${t+1}$ (with time ${t=0}$ corresponding to birth), the total expected reproductive output for the individual over his or her entire lifetime is given by the \emph{net reproductive number} ${\cR_0 := \sum_{t=0}^\infty f\,s^t = f\rb{1-s}^{-1}}$. Intuitively, $r$ and $\cR_0$ are always on the same side of $1$.

\paragraph{Matrix population model and stage structure.}

To the simple model ${\scf{x}{t+1}=r\,\scf{x}{t}}$ we can add \emph{stage structure}. Stage can refer to age (for example, individuals between one and two years old), length (for example, individuals between 5~cm and 10~cm in length), or life stage (for example, newborn, juvenile, or adult). Suppose that $\scif{x}{k}{t}$ is the number of individuals of stage $k$ at time $t$, where ${k\in\cE_m}$ and ${m\in\Nd{}}$ with ${\cE_m:=\setb{1,\ldots,m}}$ and ${\Nd{}:=\setb{1,2,\ldots}}$, is the number of stages. Then, we can form the \emph{population vector} ${\vecf{x}{t}:=\rb{\scif{x}{1}{t},\ldots,\scif{x}{m}{t}}\in\Rm}$, which is a column vector (and not a row vector, as indicated by the commas). Omitting specifics which appear later in \S\ref{sec02}, each $\scif{x}{k}{t+1}$ can be modeled as depending linearly on the populations of all stages at time $t$ with the coefficients given by appropriate fecundities and survival probabilities. That is, we can work with a model of the form ${\vecf{x}{t+1}=\mat{A}\,\vecf{x}{t}}$, with initial condition ${\vecf{x}{0}=\vec{x}_0}$, where ${\vec{x}_0\in\Rmnn}$, ${\mat{A}:=\mat{F}+\mat{S}}$, ${\mat{F},\mat{S}\in\matRmmnn}$, ${\norm{\mat{S}}<1}$, and ${t\in\Nd{}_0}$. (Since ${\scif{x}{k}{t+1}=\sum_{\ell=1}^mA_{k\ell}\,\scif{x}{\ell}{t}}$, the element $A_{k\ell}$ is interpreted as the contribution that the population of stage $\ell$ contributes to the population of stage $k$ from one time increment to the next.) Determining the proper values of the vital rates (survival probabilities, fecundities) is often called the \emph{calibration problem}, and can be expressed as a nonlinear maximization problem with linear constraints \cite{Logofet2013}. Note that $\Rmnn$ and $\matRmmnn$ respectively denote the set of all nonnegative $m$-dimensional and ${m \times m}$ matrices. Further, we are using the $L_1$-norm. See \S\ref{appA.01}, which contains a terse collection of notation, terminology, and standard results that we employ in this paper.

\paragraph{Dispersion.}

A natural extension of the model ${\vecf{x}{t+1}=\mat{A}\,\vecf{x}{t}}$ is to incorporate dispersion between a finite number of patches. These patches may be distinguished or characterized, for example, by different basins or regions where populations congregate or even political jurisdictions. If there are ${n\in\Nd{}}$ patches, we can regard ${\vecf{x}{t}\in\Rdnn{mn}}$ as the population vector with the component $\scif{x}{k+\rb{i-1}m}{t}$ being the number of individuals of demographic stage $k$ in patch $i$ at time $t$. The evolution of the population can be modeled as ${\vecf{x}{t+1}=\mat{P}\,\vecf{x}{t}}$, with initial condition ${\vecf{x}{0}=\vec{x}_0}$, where ${\vec{x}_0\in\Rdnn{mn}}$, ${\mat{P}:=\mat{D}\,\mat{A}}$, ${\mat{D},\mat{A}\in\matRdnn{mn}{mn}}$, and ${t\in\Nd{}_0}$. Here, multiplication by $\mat{A}$ performs demographic changes within each patch and multiplication by $\mat{D}$ performs dispersion between the patches. The growth rate and net reproductive number in this context are taken to be ${r:=\specrad{\mat{P}}}$ and ${\cR_0:=\specrad{\mat{N}}}$, where $\mat{N}$ is the \emph{next-generation matrix} that will be derived later and $\specrad{\cdot}$ is the spectral radius function.

\paragraph{Evolution of dispersal.}

It is not true, in general, that ${\specrad{\mat{P}}\leq\specrad{\mat{A}}}$ \cite{AxtellHanHershkowitzNeumannSze2009,JohnsonBru1990}. However, it may be the case that ${\specrad{\mat{P}}>\specrad{\mat{A}}}$ for invasive species. If so, perhaps one strategy for control is to influence the parameters so that ${\specrad{\mat{P}}\leq\specrad{\mat{A}}}$. After all, if dispersion is no longer beneficial for the growth of a population, then the population may reduce or cease the dispersion. This is related to the evolution of dispersal \cite{Greenwood-LeeTaylor2001}.

\paragraph{Goal.}

The general goal of this paper is to develop and analyze a general matrix population model that is applicable to the population dynamics of the invasive round goby fish that incorporates stage structure and dispersion. Particularly, we investigate whether dispersion can amplify the net reproductive number. More specifically, if $\cR_0'$ and $\cR_0''$ are the net reproductive numbers for the global system respectively with and without dispersion, can we have ${\cR_0'<1<\cR_0''}$ (dispersion prevents extinction)?

\paragraph{Discussion of results.}

In \S\ref{sec02}, we review standard material on matrix population models with dispersion (see, for example, \cite{Caswell2001}), including the concepts of the growth rate $r$, next-generation matrix $\mat{N}$, and net reproductive number $\cR_0$. Local demography matrices are Usher, for simplicity. For our analysis, we introduce certain submatrices, which we call newborn submatrices, whereby rows and columns of matrices corresponding to non-newborns are removed (see \S\ref{sec02}, in particular \eqref{eq02.15}, for details). Properties are developed and a second net reproductive number $\hcR_0$ is given which quantifies the maximum reproductive output of an average newborn individual. Write ${\mat{N}=\mat{D}\,\mat{W}}$, where $\mat{D}$ is the global dispersion matrix and we refer to $\mat{W}$ as the partial next-generation matrix. We show that ${\bmat{N}=\bmat{D}\cdot\bmat{W}}$, ${\specrad{\bmat{N}}=\specrad{\mat{N}}}$, and ${\specrad{\bmat{W}}=\specrad{\mat{W}}}$, where we denote by $\bmat{B}$ the newborn submatrix of matrix $\mat{B}$. Furthermore, we establish ${\cR_0\leq\hcR_0}$ and show that $\cR_0$ corresponds to the reproductive output of an individual given by the distribution characterized by the Perron eigenvector of $\bmat{N}$. In \S\ref{sec03}, detailed examples are provided. Notably, graph reduction is applied to two-patch and three-patch models with only one stage dispersing and, consequently, the net reproductive is between the smallest and largest local dispersion-free net reproductive numbers. Moreover, a two-patch example is given in which the individual populations would go extinct in isolation yet dispersion enables the total population to grow. In \S\ref{sec04}, we apply the preceding material to the invasive round goby fish. In \S\ref{sec05}, concluding remarks and open problems are stated. Finally, material that would disrupt the flow of the main narrative are located in the appendices. Standard notation, terminology, and results we use is briefly described in \S\ref{appA.01} and a more detailed review of graph reduction is presented in \S\ref{appA.02}. Unless otherwise stated, proofs are either in \S\ref{appB} or are omitted for brevity.

\section{Matrix Population Model with Dispersion} \label{sec02}

In the mathematics of population, the Leslie matrix is the fundamental model for a stage-structured population in which time is discrete. As time increments, individuals must proceed from one stage---usually interpreted as age---to the next. The Usher matrix generalizes the Leslie matrix by allowing individuals to remain within a stage during multiple time increments. In this section, we will review the Leslie and Usher matrices, incorporate dispersion, and derive the next-generation matrix and net reproductive number. Furthermore, we present results that we will use to determine if specific models allow for the amplification of the net reproductive number with the presence of dispersion.


\paragraph{Stages, patches, and population.}

Suppose a population is divided into $m$ stages, where ${m\in\setb{2,3,\ldots}}$ is fixed, and spread over $n$ patches, where ${n\in\Nd{}}$ is also fixed. The stages could be, for example, larvae, juveniles, and adults. Let $\scimf{x}{k}{i}{t}$ be the number of (female) individuals in stage $k$ in region $i$ at time $t$ (the number of males would be roughly proportional), where ${k\in\cE_m}$, ${i\in\cE_n}$, and ${t\in\Nd{}_0}$. Here, we regard time $t$ as discrete with one time increment as the minimum length of one stage (for example, one calendar year when the stages are taken to be the number of years since birth). To be biologically realistic, we assume ${\scimf{x}{k}{i}{0} \geq 0}$ for each stage $k$. We will refer to the individuals in stage ${k=1}$ as \emph{newborns} and the individuals in stage ${k>1}$ as \emph{non-newborns}. For the \emph{global population vector}, we define ${\vecf{x}{t}\in\Rdnn{mn}}$ by ${\scif{x}{k+\rb{i-1}m}{t}:=\scimf{x}{k}{i}{t}}$ for ${k\in\cE_m}$ and ${i\in\cE_n}$. Note that the \emph{local population vector} for patch ${i\in\cE_n}$ is ${\vecmf{x}{i}{t}=\rb{\scimf{x}{1}{i}{t},\ldots,\scimf{x}{m}{i}{t}}\in\Rmnn}$.

\paragraph{Fecundity and survival probability.}

Let ${\scim{f}{k}{i} \geq 0}$, where ${k\in\cE_m}$ and ${i\in\cE_n}$, be the average number of offspring born to an individual of stage $k$ in patch $i$ in a given time increment which survive to the next time increment (the \emph{fecundity}). Moreover, let ${\scim{s}{k,k}{i}\in\sqb{0,1}}$ be the probability that an individual in stage $k$ and patch $i$ will survive to the next time increment and remain in stage $k$ (the \emph{survival probability}). Similarly, let ${\scim{s}{k+1,k}{i}\in\sqb{0,1}}$ be the probability that an individual in stage $k$ and patch $i$ will survive to advance to stage ${k+1}$ for the next time increment. Since there is no stage ${m+1}$, we take ${\scim{s}{m+1,m}{i}:=0}$. We need to assume
\begin{equation} \label{eq02.01}
    0 \leq \scim{s}{k,k}{i} + \scim{s}{k+1,k}{i} < 1
    \mymbox{for all}
    k \in \cE_m
    \mymbox{and}
    i \in \cE_n
\end{equation}
to ensure that not all individuals survive until the next time increment.

\paragraph{Dispersion.}

For any stage ${k\in\cE_m}$ and pair of regions ${\rb{i,j}\in\cE_n\times\cE_n}$, we let ${\scim{d}{k}{i,j}\in\sqb{0,1}}$ be the probability that an individual of stage $k$ will disperse from region $j$ to region $i$ during a given time increment (the \emph{dispersion probability}). We assume that
\begin{equation} \label{eq02.02}
    \sum_{i=1}^n \scim{d}{k}{i,j} = 1
    \mymbox{for all}
    k \in \cE_m
    \mymbox{and}
    j \in \cE_n,
\end{equation}
meaning that any individual will either disperse to exactly one new region or remain in the same region during a given time increment.

\begin{remark}
A decision must be made regarding the order of demography and dispersion. Here, for a given time increment, we have demography (reproduction and survival) occur first within each region and then dispersion between regions occurs second.
\end{remark}

\begin{remark}
If we instead had the condition ${\sum_{i=1}^n\scim{d}{k}{i,j} \leq 1}$, then we would be allowing the possibility of death or removal during the dispersion process. We choose ${\sum_{i=1}^n\scim{d}{k}{i,j}=1}$ for mathematical convenience, but it is fairly realistic assumption when the regions are close together and/or the dispersion process is quick compared to other processes.
\end{remark}

\paragraph{Local Usher and Leslie matrices.}

For the moment, ignore dispersion and focus only on a single region, say fixed ${i\in\cE_n}$. The local demography can be modeled as ${\vecmf{x}{i}{t+1}=\matm{A}{i}\,\vecmf{x}{i}{t}}$. Since $\scim{A}{k\ell}{i}$ is the proportion that the population $\scimf{x}{\ell}{i}{t}$ contributes to the population $\scimf{x}{k}{i}{t+1}$, we have for example ${\scim{A}{11}{i}=\scim{s}{1,1}{i}+\scim{f}{1}{i}}$. We are then presented with the famous \emph{Usher matrix} (or \emph{local demography matrix})
\begin{equation} \label{eq02.03}
    \matm{A}{i} :=
    \underbrace{ \begin{pmatrix}
        \scim{f}{1}{i} & \scim{f}{2}{i} & \scim{f}{3}{i} & \cdots & \scim{f}{m-1}{i} & \scim{f}{m}{i} \\
        0 & 0 & 0 & \cdots & 0 & 0 \\
        0 & 0 & 0 & \cdots & 0 & 0 \\
        \vdots & \vdots & \vdots & \ddots & \vdots & \vdots \\
        0 & 0 & 0 & \cdots & 0 & 0 \\
        0 & 0 & 0 & \cdots & 0 & 0
    \end{pmatrix} }_{ \matm{F}{i} \in \matRmmnn }
    +
    \underbrace{ \begin{pmatrix}
        \scim{s}{1,1}{i} & 0 & 0 & \cdots & 0 & 0 \\
        \scim{s}{2,1}{i} & \scim{s}{2,2}{i} & 0 & \cdots & 0 & 0 \\
        0 & \scim{s}{3,2}{i} & \scim{s}{3,3}{i} & \cdots & 0 & 0 \\
        \vdots & \vdots & \vdots & \ddots & \vdots & \vdots \\
        0 & 0 & 0 & \cdots & \scim{s}{m-1,m-1}{i} & 0 \\
        0 & 0 & 0 & \cdots & \scim{s}{m,m-1}{i} & \scim{s}{m,m}{i}
    \end{pmatrix} }_{ \matm{S}{i} \in \matRmmnn },
\end{equation}
with $\matm{A}{i}$ decomposed into a \emph{local fecundity matrix} $\matm{F}{i}$ and a \emph{local survival matrix} $\matm{S}{i}$. When ${\scim{s}{k,k}{i}=0}$ for every ${k\in\cE_m}$, that is all individuals must advance to the subsequent stage during a time increment, then $\matm{A}{i}$ is the \emph{Leslie matrix}. Since there would be no concern for confusion, we just write ${\scim{s}{k}{i}=\scim{s}{k+1,k}{i}}$ for ${k\in\cE_{m-1}}$. Due to the form of $\matm{S}{i}$, we can rephrase the assumption \eqref{eq02.01} using \eqref{eqA.01} as ${\norm{\matm{S}{i}}<1}$.

\paragraph{Local dispersion matrices.}

To encode the dispersion between pairs of regions, we will utilize the \emph{local dispersion matrices}
\begin{equation} \label{eq02.04}
    \matm{D}{i,j} :=
    \bigoplus_{k=1}^m \scim{d}{k}{i,j}
    \in \matRmmnn
    \mymbox{for}
    i, j \in \cE_n.
\end{equation}
(Equivalently, ${\matm{D}{i,j}=\diag{\scim{d}{1}{i,j},\ldots,\scim{d}{m}{i,j}}}$ with the terms ${\setb{\scim{d}{k}{i,j}}_{k=1}^m}$ on the diagonal with all other elements being zero.) Then, the local population vector $\vecmf{x}{i}{t+1}$ is the result of dispersion to region $i$ after local demography has taken place in each of the regions. That is, ${\vecmf{x}{i}{t+1}=\sum_{j=1}^n\matm{D}{i,j}\,\matm{A}{j}\,\vecmf{x}{j}{t}}$. See Figure~\ref{fig02.01}. Note that \eqref{eq02.02} translates to ${\sum_{i=1}^n\matm{D}{i,j}=\mat{I}}$ for each ${j\in\cE_n}$.

\paragraph{Global matrices.}

We will also make use of the \emph{global fecundity matrix}, the \emph{global survival matrix}, the \emph{global Usher matrix} (or \emph{global demography matrix}), and the \emph{global dispersion matrix}, respectively given by
\begin{equation} \label{eq02.05}
    \mat{F} := \bigoplus_{i=1}^n \matm{F}{i},
    \quad
    \mat{S} := \bigoplus_{i=1}^n \matm{S}{i},
    \quad
    \mat{A} := \mat{F} + \mat{S},
    \mymbox{and}
    \mat{D} :=
    \rb{ \matm{D}{i,j} }_{i,j=1}^n,
\end{equation}
with all four matrices being in $\matRdnn{mn}{mn}$ and the first three being block-diagonal. Note that, in \eqref{eq02.05}, $\matm{D}{i,j}$ is the $\rb{i,j}^\text{th}$ block of $\mat{D}$. Proposition~\ref{thm02.04} below summarizes properties of $\mat{S}$ and $\mat{D}$, which are typical for survival and dispersion matrices.

\begin{remark}
Commonly, we will omit the prefaces ``local'' and ``global'' for brevity when there is no possibility of confusion.
\end{remark}

\begin{proposition} \label{thm02.04}
Consider the survival $\mat{S}$ and dispersion $\mat{D}$ matrices, defined in \eqref{eq02.05}.
\begin{enumerate}[label={(\alph*)},ref={\thetheorem(\alph*)},leftmargin=*,widest=a]
    \item\label{thm02.04a}
        The survival matrix satisfies ${\norm{\mat{S}}<1}$.
    \item\label{thm02.04b}
        The dispersion matrix is column stochastic, that is ${\sum_{p=1}^{mn}D_{pq}=1}$ for every ${q\in\cE_{mn}}$. Consequently, ${\norm{\mat{D}\,\vec{x}}=\norm{\vec{x}}}$ and ${\norm{\mat{D}\,\mat{X}}=\norm{\mat{X}}}$ for every ${\vec{x}\in\Rdnn{mn}}$ and ${\mat{X}\in\matRdnn{mn}{mn}}$. Moreover, ${\specrad{\mat{D}}=\norm{\mat{D}}=1}$.
\end{enumerate}
\end{proposition}

\begin{remark}
The equalities ${\norm{\mat{D}\,\vec{x}}=\norm{\vec{x}}}$ and ${\norm{\mat{D}\,\mat{X}}=\norm{\mat{X}}}$ in Proposition~\ref{thm02.04} are intuitive since left-multiplication by $\mat{D}$ can be regarded as a rearrangement of a given population with no loss of individuals.
\end{remark}

\begin{figure}[t]
    \begin{center}
        \includegraphics[scale=0.75]{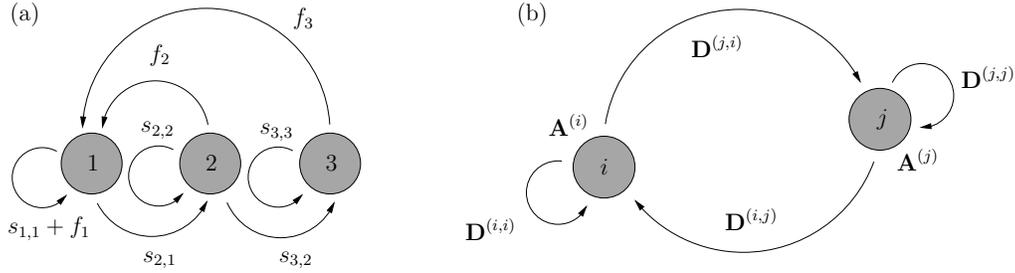}
        \caption{\textbf{(a)} Life-cycle graph for a model with ${m=3}$ stages and only one patch. (Superscripts for the patch are omitted when only one patch is under consideration.) The arrows indicate the ways in which individuals can be produced for or transition between stages. \textbf{(b)} For two regions $i$ and $j$, each with its own local demography matrix, respectively $\matm{A}{i}$ and $\matm{A}{j}$, individuals are allowed to disperse from one region to the other or remain in the same region. Here, $\matm{D}{i,j}$ is the local dispersion matrix that accounts for dispersion from $j$ to $i$.}
        \label{fig02.01}
    \end{center}
\end{figure}

\paragraph{Global model.}

Collectively, the global population is modeled by
\begin{equation} \label{eq02.06}
    \vecf{x}{t+1} = \mat{P} \, \vecf{x}{t},
    \quad
    \vecf{x}{0} = \vec{x}_0,
    \quad
    t \in \Nd{}_0,
\end{equation}
where ${\vec{x}_0\in\Rdnn{mn}}$ and ${\mat{P}:=\mat{D}\,\mat{A}=\mat{D}\,\mat{F}+\mat{D}\,\mat{S}}$ is the \emph{global projection matrix}. The component $\scif{x}{k+\rb{i-1}m}{t}$ is the number of individuals of stage ${k\in\cE_m}$ in patch ${i\in\cE_n}$ at time $t$.


\paragraph{Global growth rate.}

The solution of \eqref{eq02.06} is given explicitly by ${\vecf{x}{t}=\mat{P}^t\,\vec{x}_0}$ for ${t\in\Nd{}_0}$. (Since ${\vec{x}_0\geq\vec{0}}$ and ${\mat{P}\geq\mat{0}}$, we obtain the biologically-necessary ${\vecf{x}{t}\geq\vec{0}}$ for every ${t\in\Nd{}_0}$.) It is natural to wonder if the population will grow without bound or go extinct as time gets arbitrarily large. We will take the \emph{global growth rate} of the population governed by \eqref{eq02.06} to be ${r:=\specrad{\mat{P}}}$. If ${r<1}$, then \dsm{\lim_{t\to\infty}\norm{\vecf{x}{t}}=0} and the population will go extinct. Conversely, if ${r>1}$, then the population will \emph{usually} grow without bound.

If $\mat{A}$ is irreducible, then it follows from the Perron-Frobenius Theorem that ${r>0}$ is an eigenvalue having associated positive left and right eigenvectors, say $\vec{u}$ and $\vec{v}$, respectively. If $\vec{u}$ and $\vec{v}$ are chosen so that ${\vec{u}^*\,\vec{v}=1}$, then the \emph{sensitivity} of $r$ with respect to the parameter $A_{k\ell}$ can be computed using the formula ${\pderivslash{r}{A_{k\ell}}=u_k\,v_\ell}$ for ${k,\ell\in\cE_m}$. See, for example, \cite{CaswellWerner1978}.

\paragraph{Global net reproductive number.}

For the population described by \eqref{eq02.06}, given explicitly by ${\vecf{x}{t}=\mat{P}^t\,\vec{x}_0}$ for ${t\in\Nd{}_0}$, one may be interested in knowing the expected total number of offspring born to an average individual over their lifetime. This crucial quantity is known as the net reproductive number and is usually denoted by $\cR_0$. If ${\cR_0<1}$, we expect the population will always decay to zero. Similarly, if ${\cR_0>1}$, we expect the population will usually grow without bound. To compute $\cR_0$, we need to consider the distribution of newborns and to construct the next-generation matrix.

\paragraph{Distribution of newborns.}

Consider an average newborn individual. Now, this individual can be located in any of the $n$ regions. With ${x_{1+\rb{i-1}m}\in\sqb{0,1}}$, ${\sum_{i=1}^nx_{1+\rb{i-1}m}=1}$, and ${x_{k+\rb{i-1}m}=0}$ for ${k\in\cE_m\backslash\setb{1}}$, we can interpret $\vec{x}$ as the initial population vector for the average newborn and $x_{1+\rb{i-1}m}$ as the probability that the individual is initially located in region $i$. This leads us to consider the two sets
\begin{equation} \label{eq02.07}
    \cX := \setst{ \vec{x} \in \Rdnn{mn} }{ \norm{\vec{x}} = 1 \mymbox{and} x_p = 0 \mymbox{for all} p \not\in \cK }
    \mymbox{and}
    \cY := \setst{ \vec{x} \in \Rnnn }{ \norm{ \vec{y} } = 1 },
\end{equation}
where the set of indices
\begin{equation} \label{eq02.08}
    \cK := \setb{ 1 + \rb{ i - 1 } m }_{i=1}^n \subset \cE_{mn}
\end{equation}
corresponds to newborns. That is, $\cX$ is the set of all possible initial populations vectors of newborns and $\cY$ is the result of dropping the components corresponding to non-newborns.

\paragraph{Global next-generation matrix.}

To determine the reproductive output of a newborn ${\vec{x}\in\cX}$, we observe that the expected population vector for the offspring born to the individual between times $t$ and ${t+1}$ is given by ${\rb{\mat{D}\,\mat{F}}\rb{\mat{D}\,\mat{S}}^t\,\vec{x}}$, since the individual must first survive then disperse for each of the $t$ time increments followed by reproduction then dispersion for the last time increment. It follows that the expected population vector for the offspring born to the individual over their lifetime is given by ${\sum_{t=0}^\infty\rb{\mat{D}\,\mat{F}}\rb{\mat{D}\,\mat{S}}^t\,\vec{x}=\mat{D}\,\mat{F}\rb{\mat{I}-\mat{D}\,\mat{S}}^{-1}\vec{x}}$, where we used the geometric series (also referred to as the von Neumann series) in addition to \eqref{eqA.02} and Proposition~\ref{thm02.04a} (which guarantees ${\specrad{\mat{D}\,\mat{S}}<1}$ and hence the convergence of the series). The matrix
\begin{equation} \label{eq02.09}
    \mat{N} := \mat{D} \, \mat{F} \rb{ \mat{I} - \mat{D} \, \mat{S} }^{-1}
\end{equation}
is known as the \emph{global next-generation matrix}. Typically, the \emph{global net reproductive number} is taken to be
\begin{equation} \label{eq02.10}
    \cR_0 := \specrad{ \mat{N} }.
\end{equation}
Later, we will characterize $\cR_0$ as the expected reproductive output for a specific choice of ${\vec{x}\in\cX}$.

The Cushing-Zhou and Li-Schneider Theorems and the Fundamental Theorem of Demography all apply to the model \eqref{eq02.06}, provided the hypotheses are met. It should be noted, however, that $\mat{N}$ is typically not irreducible. Notice also
\begin{equation} \label{eq02.11}
    \cR_0 \leq \frac{ \norm{ \mat{F} } }{ 1 - \norm{ \mat{S} } },
    \mymbox{with}
    \norm{ \mat{F} } = \max{ k \in \cE_m, i \in \cE_n }{ \scim{f}{k}{i} }
    \mymbox{and}
    \norm{ \mat{S} } = \max{ k \in \cE_m, i \in \cE_n }{ \scim{s}{k,k}{i} + \scim{s}{k+1,k}{i} } < 1.
\end{equation}
Consequently, the net reproductive number cannot be made arbitrarily large by varying the dispersion rates.

\paragraph{Local growth rates and net reproductive numbers.}

With no dispersion, that is ${\mat{D}=\mat{I}}$, each region is isolated and has its own \emph{local dispersion-free growth rate} of $\scm{r}{i}$. Similarly, we can specify the \emph{local dispersion-free next-generation matrix} $\matm{N}{i}$ and the \emph{local dispersion-free net reproductive number} $\scim{\cR}{0}{i}$. Explicitly,
\begin{equation} \label{eq02.12}
    \scm{r}{i} := \specrad{\matm{A}{i}},
    \quad
    \matm{N}{i} := \matm{F}{i} \sqb{ \mat{I} - \matm{S}{i} }^{-1},
    \mymbox{and}
    \scim{\cR}{0}{i}
    := \specrad{ \matm{N}{i} }
    = \scim{N}{11}{i}
    = \sum_{k=1}^m \frac{ \scim{f}{k}{i} }{ 1 - \scim{s}{k,k}{i} } \prod_{\ell=1}^{k-1} \frac{ \scim{s}{\ell+1,\ell}{i} }{ 1 - \scim{s}{\ell,\ell}{i} }
\end{equation}
for ${i\in\cE_n}$. The formula for $\scim{\cR}{0}{i}$ is the standard formula for an Usher matrix. See, for example, \cite{CushingZhou1994}.

\begin{remark} \label{thm02.06}
One of our goals is obtain an example where ${\cR_0'<1<\cR_0''}$, where ${\cR_0':=\specrad{\mat{N}'}}$ (dispersion-free) and ${\cR_0'':=\specrad{\mat{N}''}}$ (having dispersion) with ${\mat{N}':=\mat{F}\rb{\mat{I}-\mat{S}}^{-1}}$ and ${\mat{N}'':=\mat{D}\,\mat{F}\rb{\mat{I}-\mat{D}\,\mat{S}}^{-1}}$. By virtue of \eqref{eq02.11}, we cannot expect an unbounded magnification of the net reproductive number from dispersion.
\end{remark}

\begin{remark}
If we express $\scim{\cR}{0}{i}$ in \eqref{eq02.12} as ${\scim{\cR}{0}{i}=\sum_{k=1}^m\scim{f}{k}{i}\,\scim{\pi}{k}{i}}$, then $\scim{\pi}{k}{i}$ can be interpreted as the expected number of time increments that the individual will spend in stage $k$. When ${\scim{s}{k,k}{i}=0}$ for each ${k\in\cE_m}$ and $\matm{A}{i}$ is Leslie, $\scim{\pi}{k}{i}$ corresponds to the \emph{cumulative survival probability} ${\scim{\pi}{k}{i}=\prod_{\ell=1}^{k-1}\scim{s}{\ell+1,\ell}{i}}$.
\end{remark}


\paragraph{Alternative global net reproductive number.}

In many mathematical models, the net reproductive number is characterized as the maximum reproductive output of an individual over their lifetime. To see if this is true of our $\cR_0$, first define the \emph{alternative global net reproductive number}
\begin{equation} \label{eq02.13}
    \hcR_0
    := \max{ \vec{x} \in \cX }{ \norm{ \mat{N} \, \vec{x} } }.
\end{equation}
Since ${\mat{N}=\mat{D}\,\mat{W}}$, where
\begin{equation} \label{eq02.14}
    \mat{W} := \mat{F} \rb{ \mat{I} - \mat{D} \, \mat{S} }^{-1}
\end{equation}
will be referred to as the \emph{partial global next-generation matrix}, we know from Proposition~\ref{thm02.04} that $\hcR_0$ satisfies ${\hcR_0=\maxst{\norm{\mat{W}\,\vec{x}}}{\vec{x}\in\cX}}$. It turns out that $\mat{W}$ is easier to compute than $\mat{N}$. In fact, we need only compute a particular submatrix of $\mat{W}$.

\paragraph{Submatrices.}

As far as the net reproductive numbers are concerned, we can discard each row ${p\not\in\cK}$ and column ${q\not\in\cK}$ of $\mat{W}$ of appropriate matrices when it comes to computation. Specifically, for an appropriate matrix (chosen from $\mat{D}$, $\mat{N}$, and $\mat{W}$ in this paper) we define the corresponding \emph{newborn submatrix} ${\bmat{X}=\setb{\bX_{ij}}_{i,j=1}^n\in\matRnn}$ by
\begin{equation} \label{eq02.15}
    \bX_{ij} := X_{1+\rb{i-1}m,1+\rb{j-1}m}
    \mymbox{for}
    \mat{X}\in\matRd{mn}{mn}
    \mymbox{and}
    i, j \in \cE_n.
\end{equation}
This gives us the \emph{global dispersion submatrix} $\bmat{D}$, the \emph{global next-generation submatrix} $\bmat{N}$, and the \emph{partial global next-generation submatrix} $\bmat{W}$.

\begin{proposition} \label{thm02.08}
The dispersion submatrix $\bmat{D}$, given in \eqref{eq02.15}, satisfies ${\norm{\bmat{D}\,\vec{y}}=\norm{\vec{y}}}$ and ${\norm{\bmat{D}\,\mat{Y}}=\norm{\mat{Y}}}$ for every ${\vec{y}\in\Rnnn}$ and ${\mat{Y}\in\matRnnnn}$. Moreover, $\bmat{D}$ is column stochastic with ${\bD_{ij}=\scim{d}{1}{i,j}}$ for every ${i,j\in\cE_n}$ and ${\specrad{\bmat{D}}=\norm{\bmat{D}}=1}$.
\end{proposition}

\begin{proposition} \label{thm02.09}
The alternative global net reproductive number $\hcR_0$, defined in \eqref{eq02.13}, satisfies ${\hcR_0=\norm{\bmat{W}}}$, where $\bmat{W}$ is defined in \eqref{eq02.15}.
\end{proposition}

\begin{remark}
The number of individuals born to an average individual initially in region ${j\in\cE_n}$ is given by ${\scim{\hcR}{0}{j}:=\sum_{i=1}^n\bW_{ij}}$, which we can refer to as the \emph{alternative local net reproductive number}. The alternative net reproductive number is the maximum of these, that is, ${\hcR_0=\maxst{\scim{\hcR}{0}{j}}{j\in\cE_n}}$.
\end{remark}

\paragraph{Main general results.}

We present here general results pertaining to the computation, relationships between, and the meaning of the two net reproductive numbers.

\begin{theorem} \label{thm02.11}
Consider the matrices $\mat{D}$, $\mat{N}$, and $\mat{W}$, respectively defined in \eqref{eq02.05}, \eqref{eq02.09}, and \eqref{eq02.14}, along with their newborn submatrices defined in \eqref{eq02.15}, and the net reproductive numbers $\cR_0$ and $\hcR_0$, respectively defined in \eqref{eq02.10} and \eqref{eq02.13}. Then, ${\bmat{N}=\bmat{D}\cdot\bmat{W}}$. Moreover, ${\specrad{\bmat{N}}=\specrad{\mat{N}}=\cR_0}$ and ${\specrad{\bmat{W}}=\specrad{\mat{W}}}$. Finally, ${\cR_0\leq\norm{\bmat{N}}\leq\norm{\mat{N}}}$ and ${\hcR_0=\norm{\bmat{W}}\leq\norm{\mat{W}}}$.
\end{theorem}

\begin{corollary} \label{thm02.12}
Consider the submatrices $\bmat{D}$ and $\bmat{W}$ given in \eqref{eq02.15} and the global net reproductive number $\cR_0$ given in \eqref{eq02.10}. Then, ${\bmat{D}=\mat{I}}$ if and only if ${\scim{d}{1}{i,i}=1}$ for every ${i\in\cE_n}$ (that is, newborns do not disperse). Moreover, if ${\bmat{D}=\mat{I}}$ then ${\cR_0=\specrad{\bmat{W}}}$.
\end{corollary}

\begin{theorem} \label{thm02.13}
Consider the next-generation matrix $\mat{N}$ and partial next-generation matrix $\mat{W}$, respectively given in \eqref{eq02.09} and \eqref{eq02.14}, along with their respective submatrices, given in \eqref{eq02.15}. Consider also the net reproductive numbers $\cR_0$ and $\hcR_0$, given in \eqref{eq02.10} and \eqref{eq02.13} respectively.
\begin{enumerate}[label={(\alph*)},ref={\thetheorem(\alph*)},leftmargin=*,widest=a]
    \item\label{thm02.13a}
        The net reproductive numbers satisfy ${\cR_0\leq\hcR_0}$.
    \item\label{thm02.13b}
        There exists ${\vecxi\in\cX}$ such that ${\alpha\leq\cR_0=\norm{\mat{N}\,\vecxi}\leq\beta}$, where ${\alpha:=\minst{\sum_{i=1}^n\bN_{ij}}{j\in\cE_n}}$ is the smallest absolute column sum of $\bmat{N}$ and ${\beta:=\maxst{\sum_{i=1}^n\bN_{ij}}{j\in\cE_n}}$ is the largest absolute column sum of $\bmat{N}$.
    \item\label{thm02.13c}
        If $\bmat{N}$ is irreducible with ${\veczeta\in\cY}$ being the normalized Perron right eigenvector associated with the eigenvalue ${\cR_0=\specrad{\bmat{N}}}$, then we can take ${\vecxi:=\mat{H}\,\veczeta}$, which is a nonnegative right eigenvector of $\mat{N}$. Note: The nonzero components of $\vecxi$ are ${\xi_{1+\rb{i-1}m}=\zeta_i}$ for all ${i\in\cE_n}$ and $\mat{H}$ is formally defined in Table~\ref{tabB.01}.
    \item\label{thm02.13d}
        Suppose $\lambda$ is an eigenvalue of $\bmat{N}$ or $\bmat{W}$ with associated right eigenvector ${\vec{v}\in\cY}$. Let ${\vec{u}:=\mat{H}\,\vec{v}}$. Then, $\norm{\mat{N}\,\vec{u}}$ is the expected reproductive output of the newborn individual given by initial distribution ${\vec{u}\in\cX}$.
\end{enumerate}
\end{theorem}

\begin{remark}
We interpret $\cR_0$ as the expected reproductive output of a newborn with distribution ${\veczeta:=\mat{G}\,\vecxi}$ (see Theorem~\ref{thm02.13}). Similarly, $\hcR_0$ is the maximum reproductive output of a newborn being initially in a particular single region (the one yielding the largest column sum of $\bmat{W}$).
\end{remark}

\begin{remark}
Typically, $\bmat{N}$ is irreducible but $\mat{N}$ is not. Using Theorem~\ref{thm02.13} and the Perron-Frobenius Theorem, if $\bmat{N}$ is irreducible and ${m>1}$ then $\mat{N}$ is reducible. Alternatively, appeal to \eqref{eqB.02} with ${\mat{X}=\mat{N}}$.
\end{remark}

\begin{remark}
The alternative net reproductive number $\hcR_0$ did not turn out to be as useful as the authors had hoped. However, it still has some advantages. First, $\hcR_0$ has a nice biological interpretation. Second, $\hcR_0$ is easier to compute than $\cR_0$. Third, $\hcR_0$ provides a simple upper bound on $\cR_0$, specifically ${\cR_0\leq\hcR_0=\norm{\bmat{W}}}$ whereas \eqref{eqA.02} guarantees only ${\cR_0\leq\norm{\mat{W}}}$. Finally, in the case that ${\cR_0<1<\hcR_0}$ there may be an initial population surge (for example, in the region $i$ with ${\hcR_0=\scim{\hcR}{0}{i}}$) followed by a decline, whereas there cannot be an initial surge when ${\hcR_0<1}$.
\end{remark}

\section{Examples} \label{sec03}

We will present a few examples illustrating some of the material from \S\ref{sec02}. The first will be a quick introduction to the graph-reduction method that we use later. The second will be longer and precludes the possibility of dispersion increasing the global net reproductive number (compared to the dispersion-free global net reproductive number). The final example will provide a model where dispersion in fact increases the global net reproductive number.

\subsection{Example \arabic{egnum}} \label{sec03.01}
\addtocounter{egnum}{1}

\begin{figure}[t]
    \begin{center}
        \includegraphics[scale=0.75]{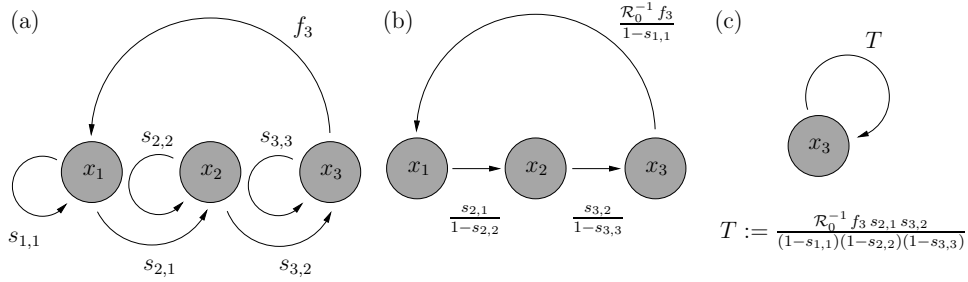}
        \caption{\textbf{(a)} Life-cycle graph for three stages with only the third stage reproducing. \textbf{(b)} A reduced $z$-transformed graph of the first, formed by multiplying the fecundity by $\cR_0^{-1}$ and thrice applying the third Mason Equivalence Rule in Figure~\ref{figA.02}. \textbf{(c)} A reduced graph of the second, formed by twice applying the second Mason Equivalence Rule in Figure~\ref{figA.02}.}
        \label{fig03.01}
    \end{center}
\end{figure}

In this example, we use graph reduction to calculate the net reproductive number for a three-stage (that is, ${m=3}$) model. Consider the life-cycle graph given in Figure~\ref{fig03.01}(a), corresponding to the Usher matrix
\[
    \mat{A} := \mat{F} + \mat{S},
    \mymbox{where}
    \mat{F} := \threebythreematrix{0}{0}{\sci{f}{3}}{0}{0}{0}{0}{0}{0}
    \mymbox{and}
    \mat{S} := \threebythreematrix{\sci{s}{1,1}}{0}{0}{\sci{s}{2,1}}{\sci{s}{2,2}}{0}{0}{\sci{s}{3,2}}{\sci{s}{3,3}}.
\]
Define, for ${\cR_0>0}$, the $z$-transformed matrix ${\mat{B}:=\cR_0^{-1}\,\mat{F}+\mat{S}}$. By the Cushing-Zhou Theorem, ${\specrad{\mat{B}}=1}$. Regard the graph of $\mat{B}$ as a system of linear equations ${\vec{x}=\mat{B}\,\vec{x}}$, with each $x_k$ as a vertex. For each ${k,\ell\in\setb{1,2,3}=\cE_3}$, we have a directed edge with rate $B_{k\ell}$ from $x_\ell$ to $x_k$ provided ${B_{k\ell} \ne 0}$. Now, rewrite the relation ${\vec{x}=\mat{B}\,\vec{x}}$ as ${\rb{\mat{B}-\mat{I}}\vec{x}=\vec{0}}$. From this equation alone it follows that $1$ is an eigenvalue of $\mat{B}$. By performing the graph-equivalent of Gaussian Elimination, specifically the Mason Equivalence Rules \cite{MasonZimmermann1960}, we can obtain an equivalent but simpler graph. See \S\ref{appA.02} for further information.

It can be shown that the graph of $\mat{B}$ is equivalent to that in Figure~\ref{fig03.01}(c) (details are given in the caption). Since it represents the equation ${x_3=Tx_3}$, obviously ${T=1}$ and thus
\begin{equation} \label{eq03.01}
    \cR_0
    = \frac{ \sci{f}{3} \, \sci{s}{3,2} \, \sci{s}{2,1} }{ \rb{ 1 - \sci{s}{1,1} } \rb{ 1 - \sci{s}{2,2} } \rb{ 1 - \sci{s}{3,3} } }.
\end{equation}
This is in agreement with formula \eqref{eq02.12}. Biologically, this is the expected total reproductive output of an average newborn, with ${\fracslash{\sci{s}{2,1}}{\rb{1-\sci{s}{1,1}}}}$ being the probability of making it the second stage, ${\fracslash{\sci{s}{3,2}}{\rb{1-\sci{s}{2,2}}}}$ being the probability of later making to the third stage, and ${\fracslash{1}{\rb{1-\sci{s}{3,3}}}}$ being the expected number of time increments at the third stage after having made it there.

\subsection{Example \arabic{egnum}} \label{sec03.02}
\addtocounter{egnum}{1}

\paragraph{General result.} Before we set up the model for this example, we will present a result which is applicable to a more broad class of models.

\begin{proposition} \label{thm03.01}
Suppose that the matrices $\mat{S}$ and $\mat{D}$, given in \eqref{eq02.05}, satisfy ${\mat{D}\,\mat{S}=\mat{S}}$ (equivalently, the local matrices satisfy ${\matm{D}{i,i}\,\matm{S}{i}=\matm{S}{i}}$ for each ${i\in\cE_n}$ and ${\matm{D}{i,j}\,\matm{S}{j}=\mat{0}}$ for each ${i,j\in\cE_n}$ with ${i \ne j}$).
\begin{enumerate}[label={(\alph*)},ref={\thetheorem(\alph*)},leftmargin=*,widest=a]
    \item\label{thm03.01a}
        The partial next-generation matrix and its submatrix satisfy ${\mat{W}=\bigoplus_{i=1}^n\matm{N}{i}}$ and ${\bmat{W}=\bigoplus_{i=1}^n\scim{\cR}{0}{i}}$, where $\matm{N}{i}$ and $\scim{\cR}{0}{i}$ are defined in \eqref{eq02.12}, $\mat{W}$ is defined in \eqref{eq02.14}, and $\bmat{W}$ is defined in \eqref{eq02.15}.
    \item\label{thm03.01b}
         The alternative net reproductive number satisfies ${\hcR_0=\max{i\in\cE_n}{\scim{\cR}{0}{i}}}$ and the next-generation submatrix satisfies ${\bN_{ij}=\scim{d}{1}{i,j}\,\scim{\cR}{0}{j}}$ for each ${i,j\in\cE_n}$, where $\hcR_0$ and $\bmat{N}$ are respectively defined in \eqref{eq02.13} and \eqref{eq02.15}.
    \item\label{thm03.01c}
        Define ${\alpha:=\min{i\in\cE_n}{\scim{\cR}{0}{i}}}$ and ${\beta:=\max{i\in\cE_n}{\scim{\cR}{0}{i}}}$. The net reproductive number satisfies ${\alpha\leq\cR_0\leq\beta}$. Furthermore, for fixed ${j\in\cE_n}$, if ${\scim{d}{1}{i,j}>0}$ for some ${i\in\cE_n}$ then ${\pderivslash{\cR_0}{\scim{\cR}{0}{j}}>0}$.
\end{enumerate}
\end{proposition}

\paragraph{Local matrices.}

Consider a three-stage (${m=3}$), multi-region scenario. Looking ahead to the application of the invasive round goby fish, we will call the three stages \emph{larvae}, \emph{juveniles}, and \emph{adults}. Suppose that the definitions of the time increment and stages are chosen so that larvae always advance to become juveniles and only the adults reproduce after one unit of time. This behaviour is captured by the local demography matrices
\begin{equation} \label{eq03.02}
    \matm{A}{i} := \matm{F}{i} + \matm{S}{i},
    \mymbox{where}
    \matm{F}{i} := \threebythreematrix{0}{0}{\scim{f}{3}{i}}{0}{0}{0}{0}{0}{0}
    \mymbox{and}
    \matm{S}{i} := \threebythreematrix{0}{0}{0}{\scim{s}{2,1}{i}}{\scim{s}{2,2}{i}}{0}{0}{\scim{s}{3,2}{i}}{\scim{s}{3,3}{i}},
\end{equation}
for ${i\in\cE_n}$. The parameters must be such that ${\matm{F}{i},\matm{S}{i}\geq\mat{0}}$ and ${\norm{\matm{S}{i}}<1}$ for each $i$. Further, suppose that only larvae disperse, so that the local dispersion matrices are
\begin{equation} \label{eq03.03}
    \matm{D}{i,j} := \diag{\scim{d}{1}{i,j},0,0}
    \mymbox{for}
    i \ne j
    \mymbox{and}
    \matm{D}{i,i} := \diag{\scim{d}{1}{i,i},1,1},
\end{equation}
where ${i,j\in\cE_n}$. Of course, we require ${\scim{d}{1}{i,j} \geq 0}$ for each ${i,j\in\cE_n}$ and ${\sum_{i=1}^n\scim{d}{1}{i,j}=1}$ for each ${j\in\cE_n}$. In this example, we consider a single patch as an input-output system among multiple connected patches. We will use graph reduction to explicitly find the global net reproductive number $\cR_0$ for the two-region case. We will also find an implicit expression for $\cR_0$ for the three-region case. A quick calculation will confirm that ${\matm{D}{i,i}\,\matm{S}{i}=\matm{S}{i}}$ for each ${i\in\cE_n}$ and ${\matm{D}{i,j}\,\matm{S}{j}=\mat{0}}$ for each ${i,j\in\cE_n}$ with ${i \ne j}$, confirming that Proposition~\ref{thm03.01} is applicable.

\paragraph{One region as an input-output system.}

\begin{figure}[t]
    \begin{center}
        \includegraphics[scale=0.75]{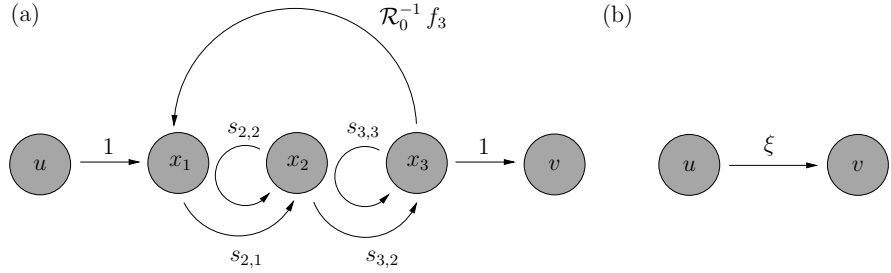}
        \caption{\textbf{(a)} The $z$-transformed life-cycle graph for an Usher model with ${m=3}$ life stages in which only the adults ${k=3}$ reproduce and the newborns ${k=1}$ cannot remain in the same stage during one time increment. \textbf{(b)} The input-output system with phantom nodes $u$ and $v$ along with rate $\xi$.}
        \label{fig03.02}
    \end{center}
\end{figure}

Consider the $z$-transformed life-cycle graph given Figure~\ref{fig03.02}(a). The local demography matrix is the same as in \eqref{eq03.02} without superscripts. Here, there is only input $u$ going to $x_1$ (the newborns) and only output $v$ leaving $x_3$ (the adults). It can be shown, using graph reduction in a manner similar to that outlined in Figure~\ref{fig03.01}, that
\begin{equation} \label{eq03.04}
    v = \xi \, u,
    \mymbox{where}
    \xi := \frac{ \sci{s}{3,2} \, \sci{s}{2,1} }{ \rb{ 1 - \sci{s}{2,2} } \rb{ 1 - \sci{s}{3,3} } - \cR_0^{-1} \, \sci{f}{3} \, \sci{s}{3,2} \, \sci{s}{2,1} }.
\end{equation}

\begin{remark}
If we allow newborns to remain newborns after one time increment, then the factor $\rb{1-\sci{s}{1,1}}$ is inserted in the obvious place in \eqref{eq03.04}. However, when ${\sci{s}{1,1}=0}$ it is easier to link multiple patches. Note that, as we can see from \eqref{eq03.01}, the denominator of $\xi$ is zero when the patch is isolated, that is ${u=0}$.
\end{remark}

\paragraph{Two regions.}

Consider now a two-patch model in which each patch is of the form presented in Figure~\ref{fig03.02}, with local demography matrices given in \eqref{eq03.02}, and only the newborns disperse, with the local dispersion matrices given in \eqref{eq03.03}. See Figure~\ref{fig03.03}(a). Here, we need to amend (by indexing everything by the patch and by replacing the fecundity by the fraction remaining in the same patch) the formula \eqref{eq03.04} and utilize
\begin{equation} \label{eq03.05}
    \xi_i := \frac{ \scim{s}{3,2}{i} \, \scim{s}{2,1}{i} }{ \sqb{ 1 - \scim{s}{2,2}{i} } \sqb{ 1 - \scim{s}{3,3}{i} } - \cR_0^{-1} \, \scim{d}{1}{i,i} \, \scim{f}{3}{i} \, \scim{s}{3,2}{i} \, \scim{s}{2,1}{i} }
\end{equation}
for each respective patch.

\begin{figure}[t]
    \begin{center}
        \includegraphics[scale=0.75]{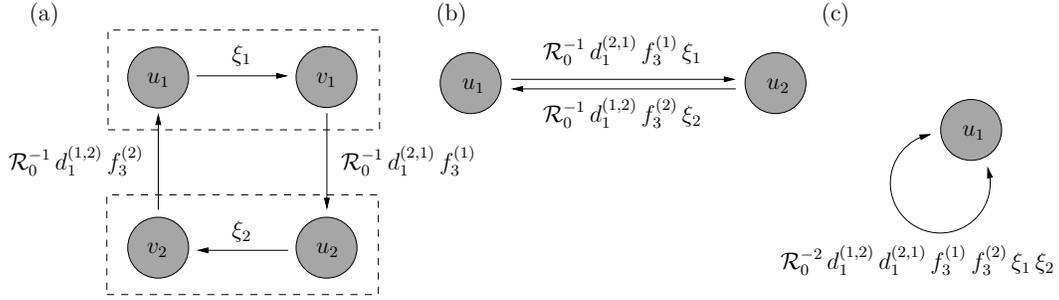}
        \caption{\textbf{(a)} The $z$-transformed life-cycle graph for a two-patch model where each patch is of the form in Figure~\ref{fig03.02}(a). \textbf{(b)} A reduced graph, formed by twice using the second Mason Equivalence Rule in Figure~\ref{figA.02}. \textbf{(c)} A further-reduced graph, also formed by using the second Mason Equivalence Rule in Figure~\ref{figA.02}.}
        \label{fig03.03}
    \end{center}
\end{figure}

Using \eqref{eq02.12}, the net reproductive numbers for the individual patches in the absence of dispersion are given by
\begin{equation} \label{eq03.06}
    \scim{\cR}{0}{i}
    = \frac{ \scim{f}{3}{i} \, \scim{s}{3,2}{i} \, \scim{s}{2,1}{i} }{ \sqb{ 1 - \scim{s}{2,2}{i} } \sqb{ 1 - \scim{s}{3,3}{i} } }.
\end{equation}
Define
\begin{equation} \label{eq03.07}
    \sci{\xi}{ij}
    := \cR_0^{-1} \, \scim{d}{1}{i,j} \, \scim{f}{3}{j} \, \xi_j
    = \frac{ \scim{d}{1}{i,j} \, \scim{\cR}{0}{j} }{ \cR_0 - \scim{d}{1}{j,j} \, \scim{\cR}{0}{j} }
    = \frac{ \cR_0^{-1} \, \bN_{ij} }{ 1 - \cR_0^{-1} \, \bN_{jj} }
    \mymbox{for}
    i \ne j,
\end{equation}
where $\xi_j$ is as in \eqref{eq03.05} and we applied \eqref{eq03.06} and Proposition~\ref{thm03.01}.

Figure~\ref{fig03.03}(c) shows a reduced graph. It follows that ${\sci{\xi}{12}\,\sci{\xi}{21}=1}$, which is an implicit algebraic equation for $\cR_0$. Hence, using \eqref{eq03.07} and the facts ${\scim{d}{1}{2,1}=1-\scim{d}{1}{1,1}}$ and ${\scim{d}{1}{1,2}=1-\scim{d}{1}{2,2}}$, we obtain
\begin{equation} \label{eq03.08}
    \cR_0^2
    - \sqb{ \scim{d}{1}{1,1} \, \scim{\cR}{0}{1} + \scim{d}{1}{2,2} \, \scim{\cR}{0}{2} } \cR_0
    + \sqb{ \scim{d}{1}{1,1} + \scim{d}{1}{2,2} - 1 } \scim{\cR}{0}{1} \, \scim{\cR}{0}{2}
    = 0.
\end{equation}
Using the Quadratic Equation (and taking the $+$ root since $\cR_0$ is the largest eigenvalue in size) to solve \eqref{eq03.08}, the global net reproductive number can be written
\begin{equation} \label{eq03.09}
    \cR_0 = \shalf \setb{ \sqb{ \scim{d}{1}{1,1} \, \scim{\cR}{0}{1} + \scim{d}{1}{2,2} \, \scim{\cR}{0}{2} }
    + \sqrt{ \sqb{ \scim{d}{1}{1,1} \, \scim{\cR}{0}{1} + \scim{d}{1}{2,2} \, \scim{\cR}{0}{2} }^2 + 4 \sqb{ 1 - \scim{d}{1}{1,1} - \scim{d}{1}{2,2} } \scim{\cR}{0}{1} \, \scim{\cR}{0}{2} } }.
\end{equation}
Special cases are presented in Table~\ref{tab03.01}.

\begin{remark} \label{thm03.03}
Suppose, in this example, that ${\scim{s}{2,2}{i}=1=\scim{s}{3,3}{i}}$ so that the local demography matrices are Leslie. Now, if we wanted to compute the global growth rate $r$ instead of the net reproductive number $\cR_0$, all instances of ${\cR_0^{-1}\,\scim{f}{3}{i}}$ should be replaced with ${r^{-1}\,\scim{f}{3}{i}}$ and all instances of $\scim{s}{k+1,k}{i}$ should be replaced with ${r^{-1}\,\scim{s}{k+1,k}{i}}$. Clearly, ${\cR_0=r^3}$. In fact, it is not hard to see that if we were considering $m$ life stages and $n$ regions with local demography matrices being Leslie, only the last-stage individuals reproducing, and only the first-stage individuals dispersing, then ${\cR_0=r^m}$.
\end{remark}

\begin{table}[t]
\begin{center}
\begin{tabular}{|c|c|c|p{9cm}|}
    \hline
    $\bm{\scim{d}{1}{1,1}}$ & $\bm{\scim{d}{1}{2,2}}$ & $\bm{\cR_0}$ & \textbf{Interpretation of Dispersion During One Time Increment} \\ \hline\hline
    0 & 0 & ${ \cR_0 = \sqrt{ \scim{\cR}{0}{1} \, \scim{\cR}{0}{2} } }$ & all newborns disperse \\ \hline
    $\half$ & $\half$ & ${ \cR_0 = \shalf \sqb{ \scim{\cR}{0}{1} + \scim{\cR}{0}{2} } }$ & all newborns disperse or remain in the same region with equal probability \\ \hline
    1 & 1 & ${ \cR_0 = \max{}{ \scim{\cR}{0}{1}, \scim{\cR}{0}{2} } }$ & no dispersion \\ \hline
    0 & 1 & $\scim{\cR}{0}{2}$ & all newborns disperse from the first region to the second without dispersion in the reverse direction \\ \hline
\end{tabular}
\caption{Special cases for \eqref{eq03.09}, all of which are consistent with Proposition~\ref{thm03.01c}.} \label{tab03.01}
\end{center}
\end{table}

\paragraph{Three regions.}

Consider now a three-patch model in which each patch is of the form presented in Figure~\ref{fig03.02} and only the newborns disperse. See Figure~\ref{fig03.04}(a). By using the same transmission rates given in \eqref{eq03.07}, we can form the reduced graph given in Figure~\ref{fig03.04}(b). Further reduction and some algebra results in the final reduced graph Figure~\ref{fig03.04}(d) and the relationship
\begin{equation} \label{eq03.10}
    \sci{\xi}{12} \, \sci{\xi}{21} + \sci{\xi}{13} \, \sci{\xi}{31} + \sci{\xi}{23} \, \sci{\xi}{32}
    + \sci{\xi}{12} \, \sci{\xi}{23} \, \sci{\xi}{31}  + \sci{\xi}{13} \, \sci{\xi}{32} \, \sci{\xi}{21}
    = 1.
\end{equation}

\begin{figure}[t]
    \begin{center}
        \includegraphics[scale=0.75]{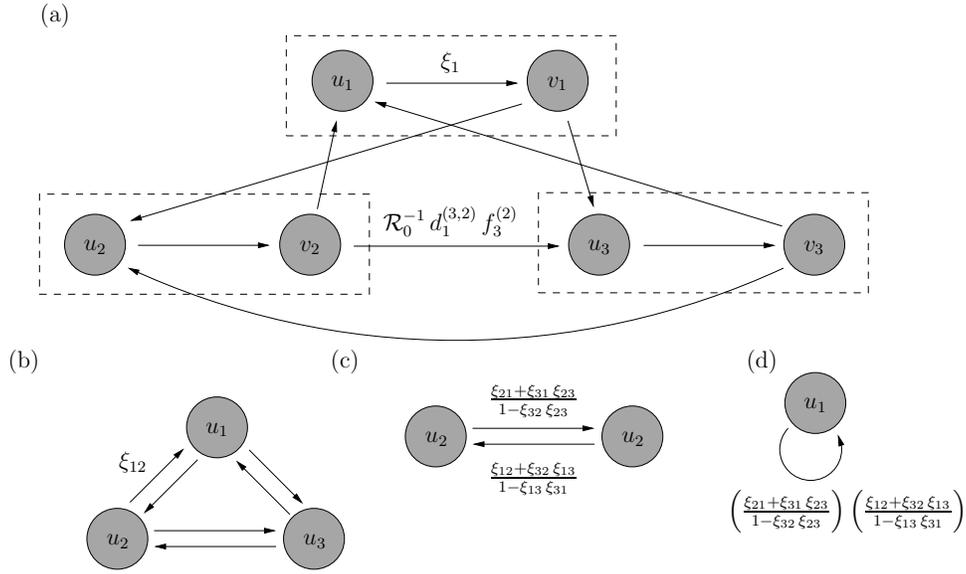}
        \caption{\textbf{(a)} The $z$-transformed life-cycle graph for a three-patch model where each patch is of the form in Figure~\ref{fig03.02}(a). Only a selection of transmission rates are included. \textbf{(b)} A reduced graph. Again, only a selection of transmission rates are included. \textbf{(c)} A further-reduced graph, formed using the first, second, and third Mason Equivalence Rules in Figure~\ref{figA.02}. \textbf{(d)} An even further-reduced graph, formed using the second Mason Equivalence Rule in Figure~\ref{figA.02}.}
        \label{fig03.04}
    \end{center}
\end{figure}

\begin{remark}
Alternatively, Equation~\eqref{eq03.10} can be obtained using the standard fact that the characteristic polynomial for a given graph can be written ${\detb{\mat{A}-\lambda\,\mat{I}}
=1+\sum_{k=1}^\infty\rb{-1}^kq_k}$, where $q_k$ is the sum of all products of loop transmissions of unordered $k$-tuples of disjoint loops (having no common nodes). For the graph in Figure~\ref{fig03.04}(b), there are five loops (namely those between each pair of nodes, three in total, plus the clockwise loop through all three nodes and the counter-clockwise loop through all three nodes) with transmission rates ${\ell_1:=\sci{\xi}{12}\,\sci{\xi}{21}}$, ${\ell_2:=\sci{\xi}{13}\,\sci{\xi}{31}}$, ${\ell_3:=\sci{\xi}{23}\,\sci{\xi}{32}}$, ${\ell_4:=\sci{\xi}{12}\,\sci{\xi}{23}\,\sci{\xi}{31}}$, and ${\ell_5:=\sci{\xi}{13}\,\sci{\xi}{32}\,\sci{\xi}{21}}$. Since there are no pairwise-disjoint loops, we see ${q_1=\sum_{k=1}^5\ell_k}$ and ${q_k=0}$ for ${k>1}$. Consequently, the characteristic equation can be written ${1+\sum_{k=1}^\infty\rb{-1}^kq_k=1+\sum_{k=1}^5\ell_k}$.
\end{remark}

\begin{remark}
Equation~\eqref{eq03.10} is decidedly more difficult to solve than the corresponding equation ${\xi_{12}\,\xi_{21}=1}$ for the two-region case. For a special case, consider ${\scim{d}{1}{i,i}=0}$ for ${i\in\setb{1,2,3}}$ and ${\scim{d}{1}{i,j}=\shalf}$ for ${i,j\in\setb{1,2,3}}$ with ${i \ne j}$. This corresponds to all newborns dispersing during one time increment with equal proportion to the two regions. Using these values and the relation \eqref{eq03.07}, after a little algebra \eqref{eq03.10} reduces to
\[
    \cR_0^3 - a \, \cR_0 - b = 0,
    \mymbox{where}
    a := \squarter \sqb{ \scim{\cR}{0}{1} \, \scim{\cR}{0}{2} + \scim{\cR}{0}{1} \, \scim{\cR}{0}{3} + \scim{\cR}{0}{2} \, \scim{\cR}{0}{3} }
    \mymbox{and}
    b := \squarter \sqb{ \scim{\cR}{0}{1} \, \scim{\cR}{0}{2} \, \scim{\cR}{0}{3} }.
\]
By virtue of Proposition~\ref{thm03.06}, stated below, ${\pderivslash{\cR_0}{\scim{\cR}{0}{i}}>0}$ for each ${i\in\setb{1,2,3}}$. Moreover, $\cR_0$ and ${a+b}$ are on the same side of $1$.
\end{remark}

\begin{proposition} \label{thm03.06}
Consider the function ${\scf{g}{u}:=u^3-au-b}$, where ${a,b>0}$ are parameters. There exists a unique positive root ${u^*>0}$ and this root satisfies ${\pderivslash{u^*}{a}>0}$ and ${\pderivslash{u^*}{b}>0}$. Moreover, any other real root ${v^*<0}$ satisfies ${u^*>-v}$. Finally, $u^*$ and ${a+b}$ are on the same side of 1.
\end{proposition}

\paragraph{Arbitrary number of regions.}

We end this example with a reduced graph, Figure~\ref{fig03.05}(a), of an $n$-patch model in which each patch is of the form in Figure~\ref{fig03.02} and only newborns disperse. To form the reduced graph, start with the relation ${\specrad{\cR_0^{-1}\,\bmat{N}}=1}$. The $z$-transformed graph is the graph formed by connecting node $j$ to node $i$ with directed arrow of magnitude ${\cR_0^{-1}\,\bN_{ij}}$ when ${\bN_{ij} \ne 0}$. Appealing to the third Mason Equivalence Rule in Figure~\ref{figA.02} and relation \eqref{eq03.07}, we can remove all self-loops and replace each ${\cR_0^{-1}\,\bN_{ij}}$ for ${i \ne j}$ with $\xi_{ij}$. The ${n=2}$ and ${n=3}$ versions appeared, respectively, in Figures~\ref{fig03.03}(b) and \ref{fig03.04}(b).

\begin{figure}[t]
    \begin{center}
        \includegraphics[scale=0.75]{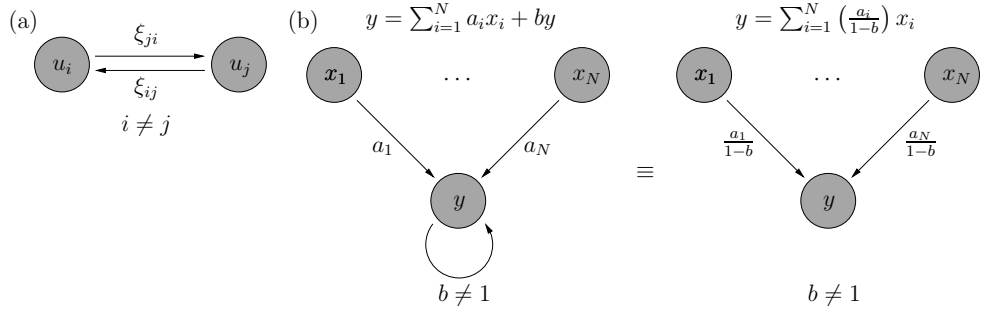}
        \caption{\textbf{(a)} The $z$-transformed life-cycle graph for an $n$-patch model where each patch is of the form in Figure~\ref{fig03.02}(a). There are no self-loops and each pair of distinct nodes $i$ and $j$ are connected only by the transmission coefficients $\xi_{ij}$ (from $j$ to $i$) and $\xi_{ji}$ (from $i$ to $j$). \textbf{(b)} A rule for obtaining the graph in (a). It is easy to verify either algebraically or by using the third Mason Equivalence Rule in Figure~\ref{figA.02}.}
        \label{fig03.05}
    \end{center}
\end{figure}

\subsection{Example \arabic{egnum}} \label{sec03.03}
\addtocounter{egnum}{1}

\paragraph{Motivation.}

Recall the main goal of this paper, which is formalized in Remark~\ref{thm02.06}. Can we find an example in which the growth rates of individual regions in the absence of dispersion are less than unity, that is ${\max{i\in\cE_n}{\scim{\cR}{0}{i}}<1}$, but the growth rate of the entire system with dispersion is greater than unity, that is ${\cR_0>1}$? Consider two regions, with the first region having high fecundities and low survival probabilities and the second region having low fecundities and high survival probabilities. Furthermore, suppose newborns disperse from the first region to the second (where they will have a better chance of survival) and adults disperse from the second region to the first (where they will have a higher reproductive output).

\paragraph{Setup.}

For a specific example, suppose ${\rb{m,n}=\rb{2,2}}$ (two stages, two patches). Here, only the adults reproduce and individuals cannot remain in the same stage after one time increment. That is, ${\scim{f}{1}{i}:=0}$, ${\scim{s}{1,1}{i}:=0}$, and ${\scim{s}{2,2}{i}:=0}$ for ${i\in\setb{1,2}}$. Moreover, assume that both regions have the same local dispersion-free net reproductive numbers of ${\scim{\cR}{0}{i}:=R}$, where ${0<R<1}$ and ${i\in\setb{1,2}}$. Furthermore, assume that the newborn-to-adult survival probability is greater in the second region than the first, say ${\scim{s}{2,1}{1}:=s}$ and ${\scim{s}{2,1}{2}:=ps}$, where ${0<s<1}$ and ${1<p<\fracslash{1}{s}}$. By virtue of \eqref{eq02.12}, the fecundities can be written ${\scim{f}{2}{1}=\fracslash{R}{s}}$ and ${\scim{f}{2}{2}=\fracslash{R}{ps}}$. Moreover, appealing to Remark~\ref{thm03.03} we know that the global growth rate and the local dispersion-free growth rates are given by ${r=\sqrt{\cR_0}}$ and ${\scm{r}{i}=\sqrt{R}}$ for ${i\in\setb{1,2}}$, where $\cR_0$ is the global net reproductive number. Finally, we assume the dispersion rate for the newborns from the first region to the second region is the same as the dispersion rate for the adults from the second region to the first region, with the common rate being ${d\in\sqb{0,1}}$.

\paragraph{Matrices.}

The local demography matrices and dispersion matrices for pairs of regions can be taken to be
\[
    \matm{A}{1} := \twobytwomatrix{0}{\frac{R}{s}}{s}{0},
    \quad
    \matm{A}{2} := \twobytwomatrix{0}{\frac{R}{ps}}{ps}{0},
    \quad
    \matm{D}{2,1} := \twobytwomatrix{d}{0}{0}{0},
    \mymbox{and}
    \matm{D}{1,2} := \twobytwomatrix{0}{0}{0}{d}.
\]
The dispersion matrices for remaining within the same region are ${\matm{D}{1,1}=\mat{I}-\matm{D}{2,1}=\diag{1-d,1}}$ and ${\matm{D}{2,2}=\mat{I}-\matm{D}{1,2}=\diag{1,1-d}}$. This leaves us with, respectively, the global projection, demography, and dispersion matrices
\[
    \mat{P} = \mat{D} \, \mat{A},
    \quad
    \mat{A} = \twobytwomatrix{ \matm{A}{1} }{ \mat{0} }{ \mat{0} }{ \matm{A}{2} },
    \mymbox{and}
    \mat{D} = \twobytwomatrix{ \matm{D}{1,1} }{ \matm{D}{1,2} }{ \matm{D}{2,1} }{ \matm{D}{2,2} }.
\]
Routine computations show the next-generation submatrix and partial next-generation submatrix to be
\[
    \bmat{N} = \twobytwomatrix{\rb{1-d}R}{pd\rb{1-d}R}{dR}{\sqb{pd^2+\rb{1-d}}R}
    \mymbox{and}
    \bmat{W} = \twobytwomatrix{R}{pdR}{0}{\rb{1-d}R}.
\]

\paragraph{Analysis.}

\begin{figure}[t]
    \centering
    \parbox{0.4\linewidth}{\includegraphics[width=\figwid]{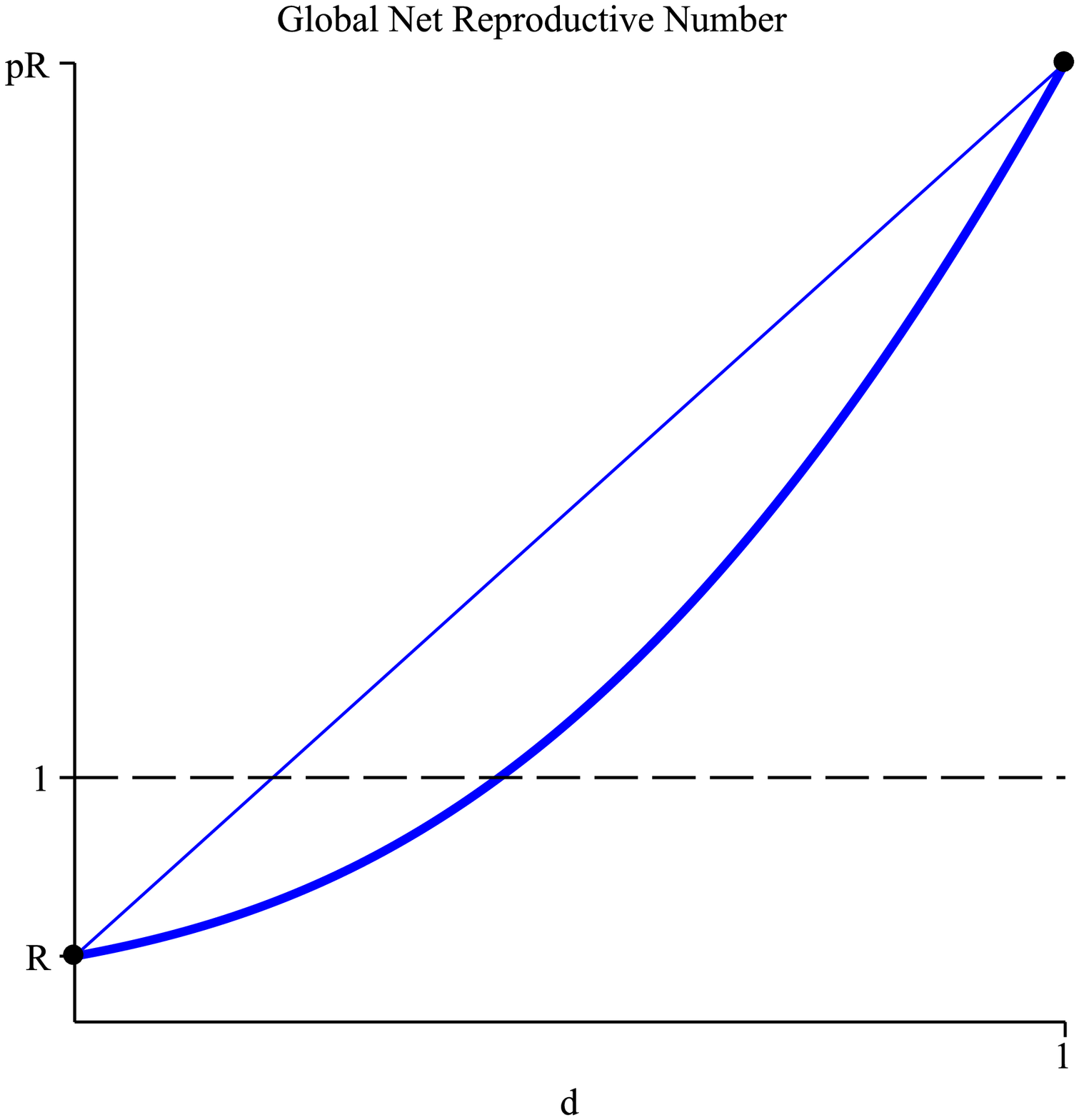}}
    \quad
    \begin{minipage}{0.4\linewidth}
        \includegraphics[width=\figwid]{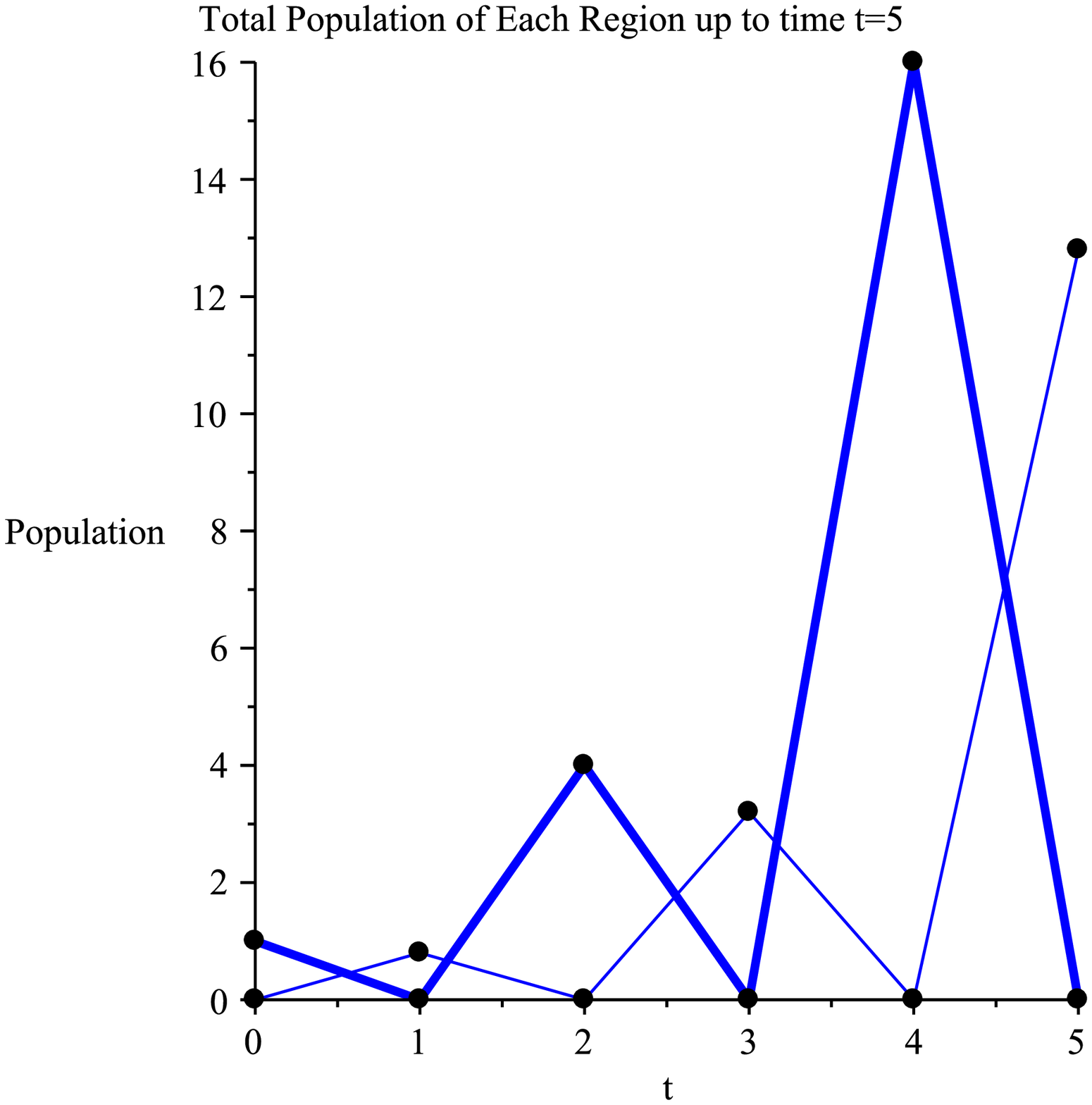}
    \end{minipage}
    \caption{Consider the example in \S\ref{sec03.03}. For the left figure, both global net reproductive numbers are given, with $\cR_0$ being the thick curve and $\hcR_0$ being the thin curve. For the right figure and the specific parameter values, the total populations for each region, namely ${\norm{\vecmf{x}{1}{t}}}$ and ${\norm{\vecmf{x}{2}{t}}}$, are given after a single newborn is placed in region 2, with region 1 being the thin curve and region 2 being the thick curve.}
    \label{fig03.06}
\end{figure}

Using Theorem~\ref{thm02.11}, we obtain the two net reproductive numbers
\[
    \scif{\cR}{0}{d} = \shalf \sqb{ 2 \rb{ 1 - d } + p \, d^2 + \sqrt{p} \, d \sqrt{ 4 \rb{ 1 - d } + p \, d^2 } } R
    \mymbox{and}
    \scif{\hcR}{0}{d} = \sqb{ \rb{ p - 1 } d + 1 } R,
\]
where we treat $d$ as the parameter of interest. From this we can glean the facts ${\scif{\cR}{0}{0}=R=\scif{\hcR}{0}{0}}$ and ${\scif{\cR}{0}{1}=pR=\scif{\hcR}{0}{1}}$.
Moreover, ${\scif{\cR'}{0}{0}=\rb{\sqrt{p}-1}R>0}$ and ${\scif{\hcR'}{0}{d}=\rb{p-1}R>0}$ for all ${d\in\sqb{0,1}}$. More tedious calculations will confirm that ${\scif{\cR'}{0}{d}>0}$ (and also ${\pderivslash{\scif{\cR}{0}{d}}{p}>0}$) for all ${d\in\sqb{0,1}}$ (and ${p>1}$). Note that the critical dispersion rate (when ${\scif{\cR}{0}{d^*}=1}$) is given by ${d^*:=\rb{1+\sqrt{\fracslash{p}{R}}}\rb{1-R}\rb{p-R}^{-1}}$.

\paragraph{Numerical example.}

Specific numerical values are also interesting. Choose ${R:=\squarter}$, ${s:=\sfrac{1}{20}}$, and ${p:=16}$, and ${d:=1}$. The critical dispersion rate is ${d^*=\sfrac{3}{7}}$. With these values, the local dispersion-free net reproductive numbers and growth rates are, respectively, ${\scim{\cR}{0}{i}=R=\squarter}$ and ${\scm{r}{i}=\sqrt{R}=\shalf}$ for ${i\in\setb{1,2}}$. However, the global net reproductive numbers and global growth rate for the entire system with dispersion are, respectively, ${\cR_0=4}$, ${\hcR_0=4}$, and ${r=\sqrt{\cR_0}=2}$. That is, ${\max{}{\scim{\cR}{0}{1},\scim{\cR}{0}{2}}<1<\cR_0=\hcR_0}$ and ${\max{}{\scm{r}{1},\scm{r}{2}}<1<r}$. So, the populations would go extinct if the regions were isolated yet they flourish if there is sufficient dispersion between the regions. See Figure~\ref{fig03.06}. Note that ${\cR_0=\norm{\mat{N}\,\vecxi}=\norm{\bmat{N}\,\veczeta}=\hcR_0}$, where ${\vecxi:=\rb{0,0,1,0}}$ and ${\veczeta:=\rb{0,1}}$ (see Theorem~\ref{thm02.13b} and its proof).

\begin{remark}
Suppose the dispersion matrices $\matm{D}{1,2}$ and $\matm{D}{2,1}$ are swapped. That is, newborns move from the second region to the first (thereby lowering their chances of survival) and the adults move from the first region to the second (thereby lowering their fecundity). It can be shown that ${\pderivslash{\cR_0}{d}<0}$ for each ${d\in\sqb{0,1}}$ with ${\cR_0=R}$ when ${d=0}$ and ${\cR_0=\fracslash{R}{p}}$ when ${d=1}$. Also, ${\hcR_0=R}$ for each ${d\in\sqb{0,1}}$.
\end{remark}

\section{Application to the Round Goby} \label{sec04}

The round goby fish, \emph{Neogobius melanostomus}, is an invasive species believed to have originated in ballast water from Eastern Europe and Western Asia that was first detected in St.~Clair River in 1990 and later known to be present in all of the Great Lakes by 2000. Since the initial introduction, the round goby has dispersed to and become well-established in many regions of the Great Lakes and surrounding tributaries. An unfortunate consequence is the decline or displacement of many native species such as the mottled sculpin, sturgeon, and trout and, interestingly, other invasive species such as zebra mussels and quagga mussels. Moreover, the goby consumes zebra mussels which negatively affect clams, crayfish, snails, and turtles. Since humans eat fish (such as smallmouth bass) which consume the round goby which eat zebra mussels which ingest toxic polychlorinated biphenyls (PCBs), the round goby fish may negatively affect human health. See, for example, \cite{Velez-EspinoKoopsBalshine2010}. The success of the round goby has been attributed to its high tolerance for a wide range of environments, diverse diet, ability to spawn repeatedly throughout the spring and summer, the aggressive protection of the eggs by the male parents, and large size compared with other benthic species of similar lifestyle \cite{CharleboisCorkumJudeKnight2001}. The Government of Ontario, like the governments of other jurisdictions, very much wants to know how to deal with the round goby---not to mention other invasive species---or, in the very least, wants to have detailed projections so as to adequately prepare for the future.

\paragraph{Considerations.}

When modeling some population, or any biological, chemical, or physical process, one must selectively choose quintessential properties and incorporate them in an appropriate way so that the resulting model both yields realistic insight and is mathematically tractable. The standard reference for quantitative analysis of fisheries is \cite{QuinnDeriso1999}. Regarding the round goby on a large scale, there are a number of essential properties which we have taken into account.

\paragraph{Fecundity.}

Sampling is used to determine the fecundity and reproductive season of the round goby. The larval stage lasts approximately three weeks and the females are sexually mature at about one year of age with spawning occurring multiple times during the spring and summer. The success of the round goby at invading the Great Lakes is commonly attributed to their high fecundity (compared to native species), extended spawning season, rapid maturation, and aggressive behaviour. In particular, the adult male gobies aggressively protect the nests \cite{MeunierYavnoAhmedCorkum2009,YavnoCorkum2011}. Data on age-length and age-fecundity relationships can be found in \cite{MacInnisCorkum2000a,MacInnisCorkum2000b,Velez-EspinoKoopsBalshine2010}.

Mathematically, these facts suggest defining a \emph{larvae stage} as individuals between zero and three weeks of age, a \emph{juvenile stage} as individuals between three weeks and one year of age, and an \emph{adult stage} as individuals aged more than one year. Moreover, adults are the only class with nonzero fecundities. The time increment, for simplicity, can be taken to be one month with six months of reproduction (April through September) and six months of no reproduction (October through March) with the larvae remaining larvae for just one time increment. Note that The accounting for seasonal changes in vital rates, for example the fecundities of the adults which are zero for non reproductive months (October through March) and positive for the reproductive months (April through September), can be achieved using periodic matrix population models \cite{CushingAckleh2012}. In regions where gobies are established, the adults tend have higher proportions of surviving offspring (fecundity).

\paragraph{Survival.}

It is estimated that round gobies have a typical lifespan of four years in the Great Lakes \cite{Velez-EspinoKoopsBalshine2010}. The juvenile gobies are known to be predatory and cannibalistic, whereby they feed on the eggs of gobies and other species. Mathematically, these facts suggest that the survival probabilities tend to be higher in regions where gobies have recently invaded.

\paragraph{Dispersion.}

The movements of round gobies can be tracked using transponder tags \cite{CookinghamRuetz2008}. Round gobies are known to have a high site fidelity but juveniles tend to disperse more rapidly than adults with larger round gobies tending to induce smaller fish to leave. Furthermore, the round goby inhabits a variety of distinct environments \cite{JohnsonAllenCorkumLee2005,RayCorkum2001}. Finally, larvae tend to migrate vertically and then disperse via water currents \cite{HaydenMiner2009,HenslerJude2007}. The large-scale dispersion is predominantly due to the larvae.

A matrix population model with the larvae having larger dispersion coefficients than the juveniles and the adults having no dispersion is appropriate. The management of aquatic populations in which dispersion is involved using matrix population models has been explored by others, for example, in \cite{KanaryLockeWatmoughChasseBourqueNadeau2011}. A small-scale model using integro-difference equations \cite{NeubertCaswell2000,RobertsonCushing2011,RobertsonCushing2012} would also be reasonable to account for the dispersion of the juveniles.

\paragraph{Linear model.}

We will formulate a matrix population model of the form explored in \S\ref{sec02} for the round goby. Nonlinear (density-dependent) and non-autonomous (time-dependent) effects are ignored. The three stages (${m=3}$) are taken to be larvae ${k=1}$, juvenile ${k=2}$, and adult ${k=3}$. For simplicity, take one time increment to represent one month so that all larvae become juveniles. With $n$ discrete patches, the local demography matrices are the same as in \eqref{eq03.02}. For the local dispersion matrices between two regions,
\[
    \matm{D}{i,j} := \diag{\scim{d}{1}{i,j},\scim{d}{2}{i,j},0}
    \mymbox{for}
    i, j \in \cE_n
    \mymbox{with}
    i \ne j,
\]
where ${0 \leq \scim{d}{k}{i,j} \leq 1}$ for each ${k\in\setb{1,2}}$ and ${i,j\in\cE_n}$ with ${i \ne j}$ and ${\sum_{i=1}^n\scim{d}{k}{i,j}=1}$ for each ${k\in\setb{1,2}}$ and ${j\in\cE_n}$. The other dispersion matrices $\matm{D}{j,j}$ for ${j\in\cE_n}$, for remaining within the same region, are obtained using the relation ${\sum_{i=1}^n\matm{D}{i,j}=\mat{I}}$. Typically, the larvae disperse more than the juveniles and so ${\scim{d}{1}{i,j}\geq\scim{d}{2}{i,j}}$ for each ${i,j\in\cE_n}$ with ${i \ne j}$. Together, the matrix population model is
\begin{equation} \label{eq04.01}
    \vecf{x}{t+1} = \mat{P} \, \vecf{x}{t},
    \quad
    \vecf{x}{0} = \vec{x}_0,
\end{equation}
where ${t\in\Nd{}_0}$, ${\vec{x}_0\in\Rdnn{3n}}$, ${\mat{P}:=\mat{D}\,\mat{A}}$, ${\mat{A}:=\bigoplus_{i=1}^n\matm{A}{i}}$, ${\mati{D}{ij}:=\matm{D}{i,j}}$ for ${i,j\in\cE_n}$. Note that $\mati{D}{ij}$ denotes the $\rb{i,j}^\text{th}$ block of matrix $\mat{D}$.

\paragraph{Two-patch example.}

This example is a special case of \eqref{eq04.01} and takes the example in \S\ref{sec03.03} one step further. Start with the local demography matrices ${\matm{A}{i}:=\matm{F}{i}+\matm{S}{i}}$ for each ${i\in\setb{1,2}}$, where $\matm{F}{i}$ and $\matm{S}{i}$ are of the form in \eqref{eq03.02}. Assume that ${\matm{S}{2}=p\,\matm{S}{1}}$, where ${p>1}$, and ignore the superscripts on the survival probabilities. The standard assumptions on the survival matrices apply, namely ${\matm{S}{1},\matm{S}{2}\geq\mat{0}}$ and ${\norm{\matm{S}{1}},\norm{\matm{S}{2}}<1}$, but we need to further assume that ${\sci{s}{2,1},\sci{s}{3,2}>0}$. Observe that
\[
    \norm{ \matm{S}{1} } = \max{}{ \sci{s}{2,1}, \sci{s}{2,2} + \sci{s}{3,2}, \sci{s}{3,3} } > 0,
    \quad
    \norm{ \matm{S}{2} } = p \norm{ \matm{S}{1} },
    \mymbox{and}
    p < \norm{ \matm{S}{1} }^{-1}.
\]
Now, take ${R\in\rb{0,1}}$. We want to choose the fecundities $\scim{f}{3}{1}$ and $\scim{f}{3}{2}$ so that the local dispersion-free net reproductive numbers are ${\scim{\cR}{0}{1}=R}$ and ${\scim{\cR}{0}{2}=R}$. This requires
\[
    \scim{f}{3}{1} := \frac{R\rb{1-\sci{s}{2,2}}\rb{1-\sci{s}{3,3}}}{\sci{s}{2,1}\,\sci{s}{3,2}}
    \mymbox{and}
    \scim{f}{3}{2} := \frac{R\rb{1-p\,\sci{s}{2,2}}\rb{1-p\,\sci{s}{3,3}}}{p^2\,\sci{s}{2,1}\,\sci{s}{3,2}}.
\]
For the local dispersion matrices,
\[
    \matm{D}{2,1} := \diag{d,0,0}
    \mymbox{and}
    \matm{D}{1,2} := \diag{0,dq,0},
    \mymbox{where}
    0 \leq d, q \leq 1.
\]
Biologically, during one time increment, a fraction of the larvae disperse from the first region to the second (when ${0<d<1}$), a smaller fraction of juveniles disperse from the second region to the first (when ${0<d<1}$ and ${0<q<1}$), larvae always advance to become juveniles, and only adults reproduce. Both regions have the same local dispersion-free net reproductive number of ${\scim{\cR}{0}{i}=R}$ for ${i\in\setb{1,2}}$, with the first region being better for reproduction and the second region being better for survival.

Explicit expressions for $\scif{\cR}{0}{d}$ and $\scif{\hcR}{0}{d}$, while somewhat messy and omitted here for that very reason, can be computed. However, we will present the maximum values of $\scif{\cR}{0}{d}$ and $\scif{\hcR}{0}{d}$,
\[
    \scif{\cR}{0}{1}
    = \xi \, R
    = \scif{\hcR}{0}{1},
    \mymbox{where}
    \xi := \frac{ p q + \rb{ 1 - q } \rb{ 1 - p \, \sci{s}{2,2} } }{ 1 - p \rb{ 1 - q } \sci{s}{2,2} }.
\]
Observe that ${\xi=p}$ when ${q=1}$. Figure~\ref{fig04.01} depicts how $\scif{\cR}{0}{d}$ and $\scif{\hcR}{0}{d}$ vary with $d$. Also shown, for two cases of $d$ (one below the critical value and one above), is the total population in each of the two regions when a single newborn is placed in the second region for the parameter values ${R:=\shalf}$, ${\sci{s}{2,1}:=\sfrac{1}{20}}$, ${\sci{s}{2,2}:=\sfrac{1}{10}}$, ${\sci{s}{3,2}:=\sfrac{1}{20}}$, ${\sci{s}{3,3}:=\sfrac{3}{20}}$, ${p:=5}$, and ${q:=\sthird}$. The critical dispersion rate, where ${\scif{\cR}{0}{d^*}=1}$, is ${d^* \approx 0.632}$.

\begin{figure}[t]
    \centering
    \parbox{0.3\linewidth}{\includegraphics[width=\figwid]{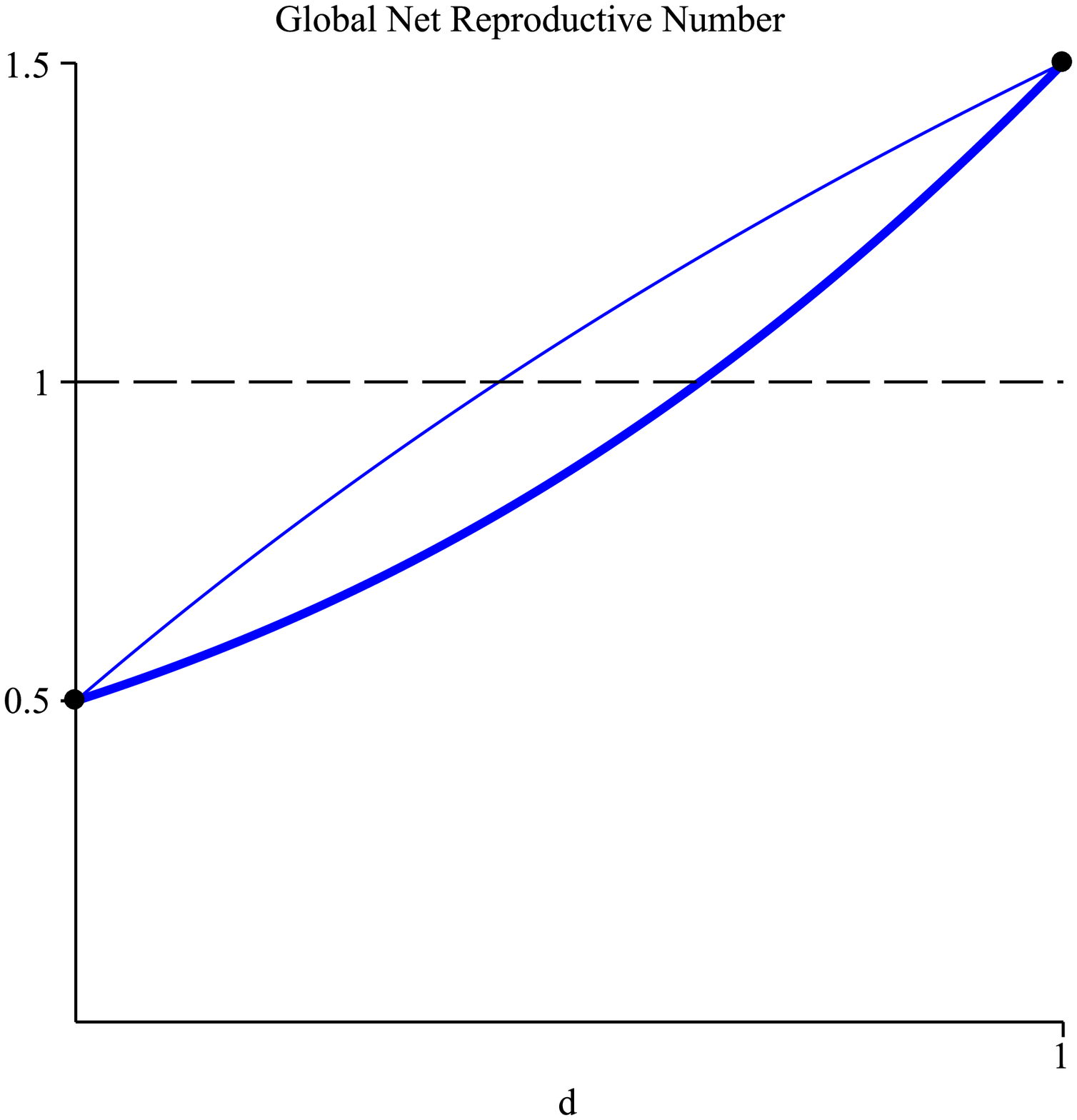}}
    \quad
    \begin{minipage}{0.3\linewidth}
        \includegraphics[width=\figwid]{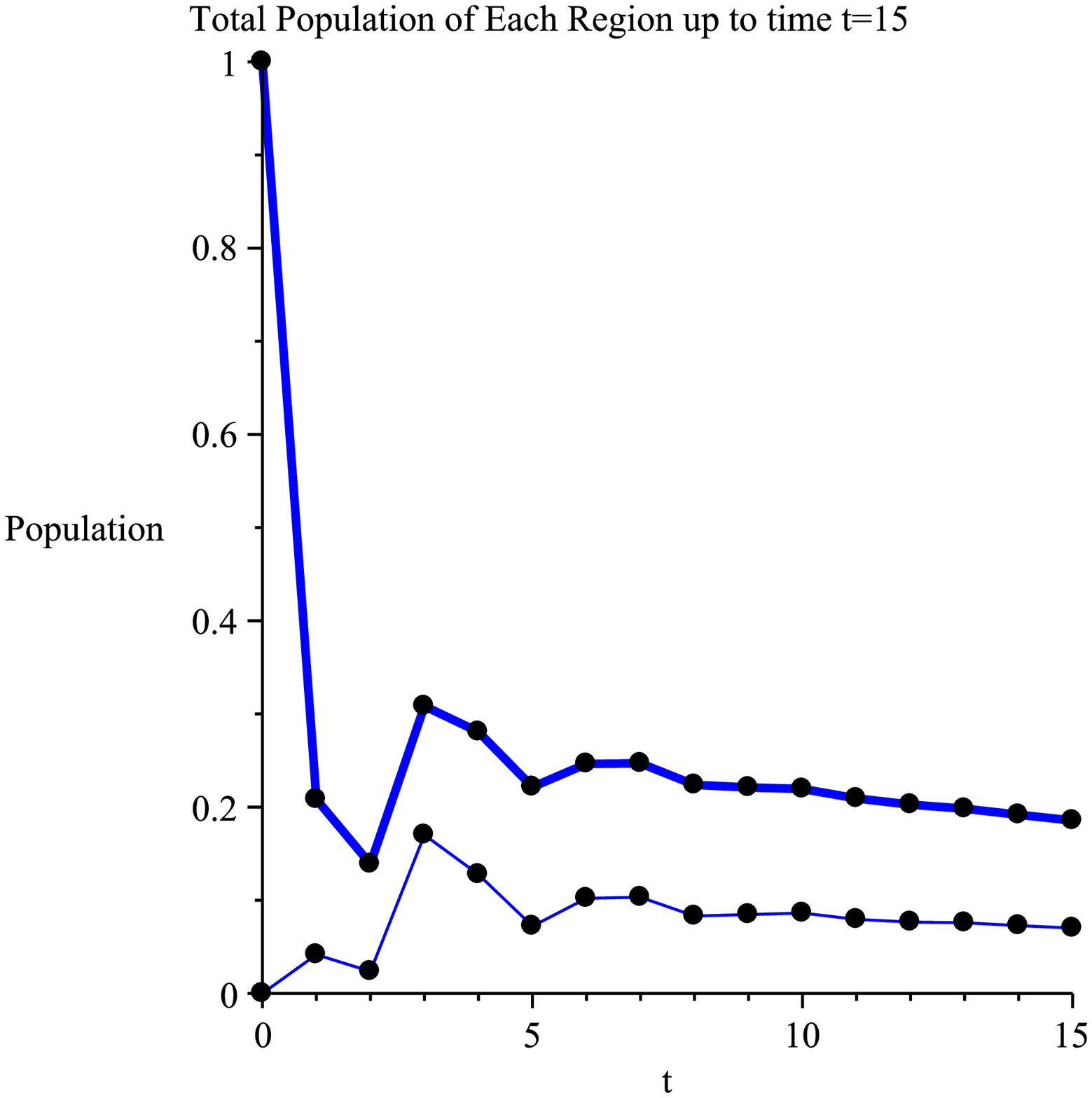}
    \end{minipage}
    \quad
    \begin{minipage}{0.3\linewidth}
        \includegraphics[width=\figwid]{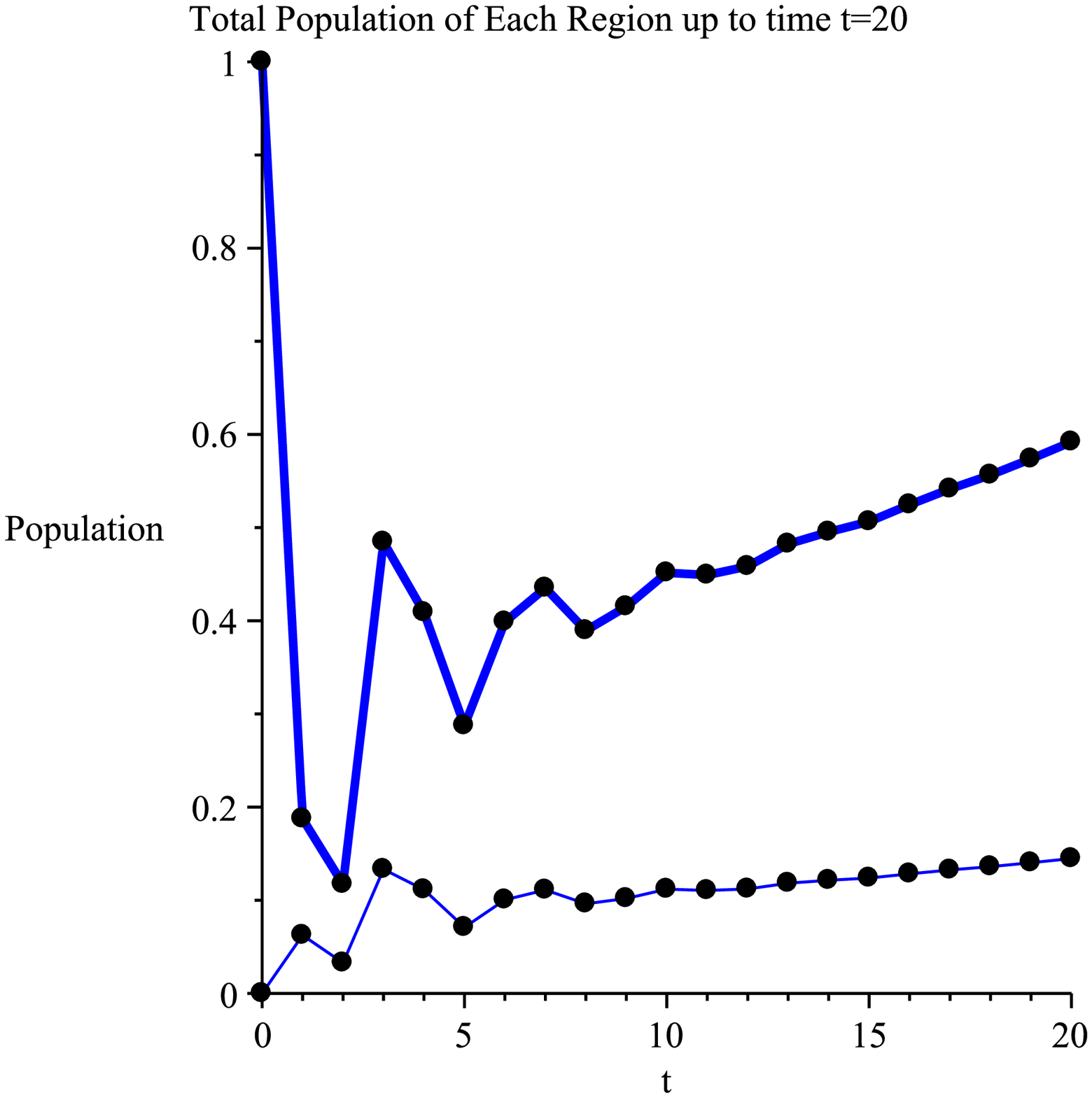}
    \end{minipage}
    \caption{Consider the numerical, two-patch example from \S\ref{sec04}. For the left figure, both global net reproductive numbers are given, with $\cR_0$ being the thick curve and $\hcR_0$ being the thin curve. For the centre figure with the specific parameter values and ${d:=\half}$, which results in ${\cR_0 \approx 0.861}$ and ${\hcR_0 \approx 1.071}$, the total populations for each region are given after a single newborn is placed in region 2, with region 1 being the thin curve and region 2 being the thick curve. For the right figure, everything is the same as the centre figure except ${d:=\frac{3}{4}}$, ${\cR_0 \approx 1.142}$, and ${\hcR_0=1.3}$.}
    \label{fig04.01}
\end{figure}

\paragraph{Discussion.}

The larvae disperse (via currents mainly) to new regions where the goby has not established and is virtually unopposed by the native species and thus has a higher chance of survival. Similarly, the juveniles are free to disperse back to other regions where the goby has already established and thus has a higher fecundity as adults. Based on the above model and analysis, we can conjecture that the combined dispersion of the larvae and juveniles amplifies the goby population. More, if the dispersion of the larvae or juveniles were to be inhibited then the amplification would be diminished or eliminated. To be sure, reduction of the other vital rates (fecundity and survival) would lower the local dispersion-free net reproductive numbers and the global net reproductive number.

\section{Conclusion and Future Work} \label{sec05}

We have presented and analyzed a matrix population model that combines multiple patches, with each region having Usher demography matrices and dispersion between the patches. Our analysis was aided by graph reduction and submatrices formed by considering only rows and columns for newborns. Biologically-relevant examples were provided, ones that were applicable to the invasive round goby fish. Notably, we showed that ``round-trip dispersion'' (which applies to the goby) can amplify the overall growth rate of a population whereas the absence of round-trip dispersion precludes such amplification.

To conclude, we state a couple of interesting questions. First, for an otherwise fixed setup, what dispersion matrix $\mat{D}$ maximizes and minimizes $\cR_0$? Second, under what conditions is it true that ${\max{i\in\cE_n}{\scim{\cR}{0}{i}}<\cR_0}$?

\appendix

\makeatletter
\def\@seccntformat#1{\csname Pref@#1\endcsname \csname the#1\endcsname\quad}
\def\Pref@section{Appendix~}
\makeatother

\section{Background Material} \label{appA}

\subsection{Common Notation, Terminology, and Results} \label{appA.01}

\paragraph{Nonnegative vectors and matrices.}

The set $\Rmnn$ consists of all $m$-dimensional (column) vectors with nonnegative components. Similarly, the set $\matRmmnn$ consists of $m$ by $m$ matrices with nonnegative elements. If dimension is clear, we can also write ${\vec{x}\geq\vec{0}}$ and ${\mat{B}\geq\mat{0}}$ to indicate that each component of $\vec{x}$ and each element of $\mat{B}$ is nonnegative. (A strict inequality says all elements are strictly positive.) Note that ${\mat{B}\geq\mat{0}}$ does \emph{not} indicate here that $\mat{B}$ is positive semi-definite.

\paragraph{Norms.}

The vector norm we use is the $L_1$-norm, since it outputs total population for a population vector $\vecf{x}{t}$. Specifically, the vector norm and its induced matrix norm are, for ${\vec{x}\in\Rm}$ and ${\mat{B}\in\matRmm}$,
\begin{equation} \label{eqA.01}
    \norm{ \vec{x} } := \sum_{k=1}^m \abs{ x_k }
    \mymbox{and}
    \norm{\mat{B}} := \max{\norm{\vec{x}}=1}{ \norm{\mat{B}\,\vec{x}} } = \max{\ell\in\cE_m}{ \sum_{k=1}^m \abs{B_{k\ell}} }.
\end{equation}

\paragraph{Irreducible and primitive matrices.}

Important subclasses of nonnegative matrices are the \emph{irreducible} and \emph{primitive} matrices. These terms have more formal and intuitive definitions using linear algebra and graph theory, but the simplest computational necessary and sufficient conditions for irreducibility and primitivity are, respectively, ${\rb{\mat{I}+\mat{B}}^{m-1}>\mat{0}}$ and ${\mat{B}^{\rb{m-1}^2+2}>\mat{0}}$, where ${B_{k\ell}:=\sgn{A_{k\ell}}}$.

\paragraph{Eigenvalues and eigenvectors.}

For a general square matrix ${\mat{A}\in\matCmm}$, if ${\lambda\in\Cd{}}$ and ${\vec{u},\vec{v}\in\Cmnz}$ satisfy ${\vec{u}^*\,\mat{A}=\lambda\,\vec{u}^*}$ and ${\mat{A}\,\vec{v}=\lambda\,\vec{v}}$, then $\lambda$ is an \emph{eigenvalue}, $\vec{u}$ is an \emph{associated left eigenvector}, and $\vec{v}$ is an \emph{associated right eigenvector}. Note $*$ denotes transposition and ${\Cmnz:=\setst{\vec{z}\in\Cm}{\vec{z}\ne\vec{0}}}$. The set of all eigenvalues of $\mat{A}$ is called the \emph{spectrum} and denoted $\spectrum{\mat{A}}$. The eigenvalues of $\mat{A}$ are found as the $n$ roots (not necessarily distinct) of the \emph{characteristic polynomial} ${\scf{p}{\lambda}:=\detb{\lambda\,\mat{I}-\mat{A}}}$. Simple-but-useful properties follow from the observations ${\mat{A}^*\,\vec{u}=\lambda\,\vec{u}}$, ${\spectrum{\mat{A}^*}=\spectrum{\mat{A}}}$, ${\mat{A}\rb{\alpha\,\vec{v}}=\lambda\rb{\alpha\,\vec{v}}}$, and ${\rb{\alpha\,\mat{A}}\vec{v}=\rb{\alpha\,\lambda}\vec{v}}$, where ${\alpha \ne 0}$.

\paragraph{Spectral radius.}

Also important is the \emph{spectral radius} of $\mat{B}$, given by ${\specrad{\mat{B}}:=\maxst{\abs{\lambda}}{\lambda\in\spectrum{\mat{B}}}}$. The spectral radius satisfies ${\specrad{\alpha\,\mat{B}}=\abs{\alpha}\specrad{\mat{B}}}$, for any ${\alpha\in\Cd{}}$, and the inequalities of Frobenius
\begin{equation} \label{eqA.02}
    \min{ \ell \in \cE_m }{ \sum_{k=1}^m \abs{B_{k\ell}} }
    \leq \specrad{\mat{B}}
    \leq \norm{ \mat{B} }
    = \max{ \ell \in \cE_m }{ \sum_{k=1}^m \abs{B_{k\ell}} }.
\end{equation}

\paragraph{Cushing-Zhou and Li-Schneider Theorems.}

Suppose that ${\mat{A}\in\matRmmnn}$ is a projection matrix and consider the matrix population model ${\vecf{x}{t+1}=\mat{A}\,\vecf{x}{t}}$, with ${\vecf{x}{0}=\vec{x}_0}$, where ${t\in\Nd{}_0}$ and ${\vec{x}_0\in\Rmnn}$. Suppose further that ${\mat{A}=\mat{F}+\mat{S}}$, where ${\mat{F},\mat{S}\in\matRmmnn}$ and ${\norm{\mat{S}}<1}$. Consider the growth rate ${r:=\specrad{\mat{A}}}$ and the net reproductive number ${\cR_0:=\specrad{\mat{N}}}$, where ${\mat{N}:=\mat{F}\rb{\mat{I}-\mat{S}}^{-1}}$ is the next-generation matrix. If ${\cR_0>0}$ and $\mat{A}$ and $\mat{N}$ are irreducible, then either ${0<\cR_0<r<1}$, or ${r=1=\cR_0}$, or ${1<r<\cR_0}$. Moreover, ${\specrad{\cR_0^{-1}\,\mat{F}+\mat{S}}=1}$ and ${\specrad{\mat{F}+\cR_0\,\mat{S}}=\cR_0}$. This is the Cushing-Zhou Theorem, and it is useful because both $r$ and $\cR_0$ have intuitive meanings but $\cR_0$ is typically easier to compute. A weaker version, which is known as the Li-Schneider Theorem and appears as Theorem~3.3 of \cite{LiSchneider2002}, guarantees that either ${0 \leq \cR_0 \leq r < 1}$, or ${r=\cR_0=1}$, or ${1 \leq r \leq \cR_0}$ in the case that $\mat{A}$ and $\mat{N}$ are assumed to just be nonnegative.

\paragraph{Perron-Frobenius Theorem.}

The proof of the Cushing-Zhou Theorem and some results in this paper rely on the Perron-Frobenius Theorem. See, for example, \cite{Bellman1970,Gantmacher1959,HornJohnson2013,LiSchneider2002}. This powerful result applies to an irreducible matrix ${\mat{A}\in\matRmmnn}$ and asserts the following: There exists an eigenvalue $r$ of $\mat{A}$ (the \emph{Perron eigenvalue}) that is real, positive, and simple (that is, the eigenvalue is not a repeated root of the characteristic polynomial), and satisfies ${r=\specrad{\mat{A}}}$; the eigenspace is one-dimensional (only one linearly-independent eigenvector) and the associated left and right eigenvectors (the \emph{Perron eigenvectors}) can be taken to be positive; no eigenvalue other than $r$ has associated eigenvectors that are positive; and the spectral radius ${r=\specrad{\mat{A}}}$ satisfies ${\pderivslash{r}{A_{k\ell}}>0}$ for each ${k,\ell\in\cE_m}$. If, in addition, $\mat{A}$ is primitive, then the eigenvalue $r$ is dominant, that is any other eigenvalue $\lambda$ satisfies ${\abs{\lambda}<r}$.

\paragraph{Fundamental Theorem of Demography.}

Suppose the population $\vecf{x}{t}$ is modeled as ${\vecf{x}{t+1}=\mat{A}\,\vecf{x}{t}}$, with ${\vecf{x}{0}=\vec{x}_0}$, and ${r:=\specrad{\mat{A}}}$. If $\mat{A}$ is primitive and $\vec{u}$ and $\vec{v}$ are, respectively, associated positive left and right (Perron) eigenvectors that are normalized so that ${\vec{u}^*\,\vec{v}=1}$, then the population satisfies ${\vecf{x}{t} \sim r^t\,\vec{u}^*\,\vec{x}_0\,\vec{v}}$ and, provided ${\vec{u}^*\,\vec{x}_0\,\vec{v}\ne\vec{0}}$, the \emph{stable-age distribution} satisfies ${\fracslash{\vecf{x}{t}}{\norm{\vecf{x}{t}}}\sim\fracslash{\vec{v}}{\norm{\vec{v}}}}$ as ${t\to\infty}$. This is the Fundamental Theorem of Demography.

\subsection{Graph Reduction} \label{appA.02}

For further information on the technique of \emph{graph reduction}, which involves finding equivalent graphs to find eigenvalues and eigenvectors, see \cite{de-Camino-BeckLewis2007,MasonZimmermann1960}.

\paragraph{Directed and life-cycle graphs.}

For a matrix ${\mat{A}\in\matRmmnn}$, the associated \emph{directed graph} is constructed by forming $m$ nodes, labeled as elements in $\cE_m$, with a directed edge from node $\ell$ to node $k$ if ${A_{k\ell}>0}$. The directed graph is \emph{strongly connected} if for every pair of nodes $\rb{k,\ell}$ there exists a directed path connecting $k$ to $\ell$ and a directed path connecting $\ell$ to $k$. Note that the undirected graph is \emph{weakly connected} if an undirected path exists between any two nodes. Importantly for us, a directed graph correspond to a \emph{life-cycle graph} for a stage-structured population. Moreover, if the directed graph associated with $\mat{A}$ is strongly connected, then $\mat{A}$ is irreducible. Two sufficient conditions for primitivity are (i) $\mat{A}$ has a positive diagonal element and (ii) $\mat{A}$ is irreducible and there is at least one self-loop in the graph.

\paragraph{Linear signal-flow graphs.}

\begin{figure}[t]
    \begin{center}
        \includegraphics[scale=0.75]{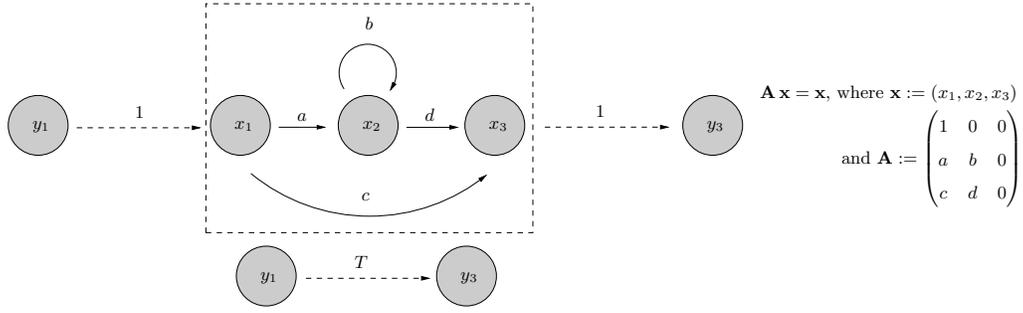}
        \caption{Graph representing the linear equations ${x_2=ax_1+bx_2}$ and ${x_3=cx_1+dx_2}$. Also shown are the phantom nodes $y_1$ (input) and $y_3$ (output) which satisfy ${y_3=Ty_1}$. Observe that ${T=\fracslash{y_3}{y_1}=\fracslash{x_3}{x_1}}$.} \label{figA.01}
    \end{center}
\end{figure}

A \emph{linear signal-flow graph}, which has origins in electronic-circuit theory, uses directed graphs representing a system of linear equations. See, for example, \cite{Caswell2001,de-Camino-BeckLewis2007,MasonZimmermann1960}. Consider a system of linear equations of the form ${\vec{x}=\mat{A}\,\vec{x}}$, where ${\vec{x}\in\Rm}$ and ${\mat{A}\in\matRmm}$. This can be written ${x_k=\sum_{\ell=1}^mA_{k\ell}\,x_\ell}$ for each ${k\in\cE_m}$. The graph is constructed as follows. There is a \emph{node} (or \emph{vertex}) for each variable $x_k$. Furthermore, if ${A_{k\ell} \ne 0}$, then a directed \emph{branch} (or \emph{edge}) is drawn from $x_\ell$ to $x_k$ with \emph{transmission} (or \emph{rate} or \emph{transmission rate}) $A_{k\ell}$. If ${A_{kk} \ne 0}$, then a \emph{self-loop} is drawn. Relations of the form ${x_k=x_k}$ (corresponding to row $k$ of $\mat{A}$ consisting of only zeros except for ${A_{kk}=1}$) are typically omitted from the graph. Note that the graph is in ``cause-and-effect'' form, with each dependent node expressed once as the effect of other cause nodes. An example is presented in Figure~\ref{figA.01}.

\paragraph{Graphs and eigenvalues.}

The \emph{$z$-transformed graph} (the name has engineering origins) of the matrix ${\mat{A}\in\matRmm}$ is the graph of ${\lambda^{-1}\,\mat{A}}$. This represents the system of equations ${\lambda^{-1}\,\mat{A}\,\vec{x}=\vec{x}}$, that is, ${\mat{A}\,\vec{x}=\lambda\,\vec{x}}$. So, $\lambda$ is a nonzero eigenvalue of $\mat{A}$ with associated eigenvector $\vec{x}$. Rules, informally known as \emph{Mason Equivalence Rules} (some of which appear in Figure~\ref{figA.02}), allow us to transform the graph and eliminate intermediate nodes, resulting in a simpler equation for the eigenvalues. With an appropriate sequence of elementary row operations (that is, Gaussian elimination) performed on ${\mat{A}-\lambda\,\mat{I}}$ resulting in ${\mat{E}\rb{\mat{A}-\lambda\,\mat{I}}}$ with ${\detb{\mat{E}} \ne 0}$, the polynomials $\detb{\mat{A}-\lambda\,\mat{I}}$ and $\detb{\mat{E}\rb{\mat{A}-\lambda\,\mat{I}}}$ have the same roots (eigenvalues).

\begin{figure}[t]
    \begin{center}
        \includegraphics[scale=0.75]{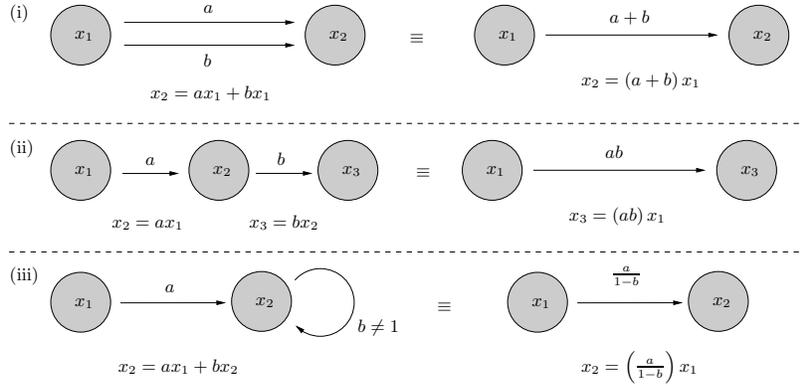}
        \caption{Three Mason Equivalence Rules.} \label{figA.02}
    \end{center}
\end{figure}

\paragraph{Graphs and the growth rate and net reproductive number.}

To address the application to population models, let $\mat{F}$, $\mat{S}$, $\mat{A}$, $\mat{N}$, $r$, and $\cR_0$ be as in the statement of the Cushing-Yicang Theorem as it appears in \S\ref{appA.01}. Assume ${r,\cR_0>0}$. Since ${\specrad{r^{-1}\,\mat{A}}=1}$ and ${\specrad{\cR_0^{-1}\,\mat{F}+\mat{S}}=1}$, performing graph reduction for the respective graphs of ${r^{-1}\,\mat{A}}$ and ${\cR_0^{-1}\,\mat{F}+\mat{S}}$ can yield $r$ and $\cR_0$. In both cases, if the characteristic polynomial has more than one root, we choose the largest root in absolute value.

\section{Proofs} \label{appB}


\begin{table}[t]
\begin{center}
\begin{tabular}{|c|p{4cm}|p{4.5cm}|p{4.5cm}|}
    \hline
    \textbf{Matrix} & \centering{\textbf{Definition}} & \textbf{Left-Multiplication} & \textbf{Right-Multiplication} \\ \hline\hline
    ${\mat{L}\in\matRdnn{mn}{mn}}$ & \centering{ For ${p,q\in\cE_{mn}}$, if ${p\in\cK}$ and ${q=p}$, then ${L_{pq}:=1}$. Otherwise, ${L_{pq}:=0}$. } & Sets to zero the rows corresponding to non-newborns. & Sets to zero the columns corresponding to non-newborns. \\ \hline
    ${\mat{G}\in\matRdnn{n}{mn}}$ & \centering{ For ${i\in\cE_n}$ and ${q\in\cE_{mn}}$, if ${q=1+\rb{i-1}m}$, then ${G_{iq}:=1}$. Otherwise, ${G_{iq}:=0}$. } & Removes the rows corresponding to non-newborns. & Inserts columns of zeros corresponding to non-newborns. \\ \hline
    ${\mat{H}\in\matRdnn{mn}{n}}$ & \centering{ ${\mat{H}:=\mat{G}^*}$ } & Inserts rows of zeros corresponding to non-newborns. & Removes the columns corresponding to non-newborns. \\ \hline
\end{tabular}
\caption{The auxiliary matrices $\mat{L}$, $\mat{G}$, and $\mat{H}$. For example, take ${m:=3}$, ${n:=2}$, ${\vec{x}:=\rb{x_1,x_2,x_3,x_4,x_5,x_6}}$, and ${\vec{y}:=\rb{y_1,y_2}}$. Then, ${\cK=\setb{1,4}}$. Moreover, $\mat{L}$ is a ${6 \times 6}$ matrix with zeros everywhere except for ${L_{11}=1}$ and ${L_{44}=1}$. Similarly,  $\mat{G}$ is a ${2 \times 6}$ matrix with zeros everywhere except for ${G_{11}=1}$ and ${G_{24}=1}$ and $\mat{H}$ is a ${6 \times 2}$ matrix with zeros everywhere except for ${H_{11}=1}$ and ${H_{42}=1}$. Then, ${\mat{L}\,\vec{x}=\rb{x_1,0,0,x_4,0,0}}$, ${\mat{G}\,\vec{x}=\rb{x_1,x_4}}$, and ${\mat{H}\,\vec{y}=\rb{y_1,0,0,y_2,0,0}}$.} \label{tabB.01}
\end{center}
\end{table}

\paragraph{Auxiliary matrices.} We employ the auxiliary matrices $\mat{L}$, $\mat{G}$, and $\mat{H}$ to handle many technical details of the proofs of results in the main text of this paper. The definitions and actions of these matrices are presented in Table~\ref{tabB.01}. The newborn submatrices, formed by expunging the rows and columns corresponding to non-newborns and defined in \eqref{eq02.15}, can be computed using $\mat{G}$ and $\mat{H}$. Specifically, ${\bmat{X}=\matPhif{\mat{X}}}$, where
\begin{equation} \label{eqB.01}
    \matPhif{\mat{X}} := \mat{G} \, \mat{X} \, \mat{H} \in \matRnn
    \mymbox{for}
    \mat{X} \in \matRd{mn}{mn}.
\end{equation}

\begin{proposition} \label{thmB.01}
Consider the sets $\cX$ and $\cY$, defined in \eqref{eq02.07}, auxiliary matrices $\mat{L}$, $\mat{G}$, and $\mat{H}$, defined in Table~\ref{tabB.01}, and the matrix function $\matPhi$, defined in \eqref{eqB.01}.
\begin{enumerate}[label={(\alph*)},ref={\thetheorem(\alph*)},leftmargin=*,widest=a]
    \item\label{thmB.01a}
        The auxiliary matrices satisfy ${\mat{G}\,\mat{H}=\mat{I}}$, ${\mat{H}\,\mat{G}=\mat{L}}$, ${\mat{G}\,\mat{L}=\mat{G}}$, and ${\mat{L}\,\mat{H}=\mat{H}}$.
    \item\label{thmB.01b}
        The matrices $\mat{N}$ and $\mat{W}$, given respectively in \eqref{eq02.09} and \eqref{eq02.14}, satisfy ${\mat{L}\,\mat{N}=\mat{N}}$ and ${\mat{L}\,\mat{W}=\mat{W}}$.
    \item\label{thmB.01c}
        For arbitrary matrix ${\mat{X}\in\matRd{mn}{mn}}$, define ${\bmat{X}:=\matPhif{\mat{X}}}$. Then, ${\mat{L}\,\mat{X}\,\mat{L}=\mat{H}\,\bmat{X}\,\mat{G}}$, ${\mat{L}\,\mat{X}\,\mat{H}=\mat{H}\,\bmat{X}}$, and ${\mat{G}\,\mat{X}\,\mat{L}=\bmat{X}\mat{G}}$.
    \item\label{thmB.01d}
        If ${\vec{x}\in\cX}$ then ${\mat{G}\,\vec{x}\in\cY}$. Similarly, if ${\vec{y}\in\cY}$ then ${\mat{H}\,\vec{y}\in\cX}$.
    \item\label{thmB.01e}
        Let ${\mat{B}\in\matRdnn{mn}{mn}}$, ${\vec{x}\in\cX}$, and ${\vec{y}\in\cY}$. Define ${\bmat{B}:=\matPhif{\mat{B}}\geq\mat{0}}$. If ${\mat{L}\,\mat{B}=\mat{B}}$, ${\vec{y}=\mat{G}\,\vec{x}}$, and ${\vec{x}=\mat{H}\,\vec{y}}$, then ${\norm{\mat{B}\,\vec{x}}=\norm{\bmat{B}\,\vec{y}}}$.
\end{enumerate}
\end{proposition}

\begin{proof}{Proposition~\ref{thmB.01}}
\begin{enumerate}[label={(\alph*)},leftmargin=*,widest=a]
    \item
        The proof is straight-forward and omitted.
    \item
        Due to the structure of $\mat{F}$ (and each component matrix $\matm{F}{i}$), we see ${N_{pq}=0}$ and ${W_{pq}=0}$ for ${p\not\in\cK}$. Thus, the action of left-multiplication by $\mat{L}$ is to set to zero rows for the non-newborns that are already zero.
    \item
        For the first relation, multiply the equation ${\bmat{X}=\mat{G}\,\mat{X}\,\mat{H}}$ on the left by $\mat{H}$ and on the right by $\mat{G}$ and apply the relation ${\mat{H}\,\mat{G}=\mat{L}}$, given in part~\ref{thmB.01a}, twice. The second and third relations are similarly obtained.
    \item
        The statement is obvious.
    \item
        First, note ${\norm{\mat{B}\,\vec{x}}=\sum_{p,q=1}^{mn}B_{pq}\,x_q}$ and ${\norm{\bmat{B}\,\vec{y}}=\sum_{i,j=1}^n\bB_{ij}\,y_j}$. We are given ${B_{pq}=0}$ and ${x_q=0}$ for ${p\not\in\cK}$ and ${q\not\in\cK}$. Consequently, ${B_{pq}\,x_q}$ can only potentially be nonzero for ${p\in\cK}$ and ${q\in\cK}$. It follows from the definition of $\cK$ that ${\sum_{p,q=1}^{mn}B_{pq}\,x_q=\sum_{i,j=1}^n\bB_{ij}\,y_j}$, as desired.
\end{enumerate}
\end{proof}

\begin{proposition} \label{thmB.02}
Consider the auxiliary matrix $\mat{L}$, defined in Table~\ref{tabB.01}, and the matrix function $\matPhi$, defined in \eqref{eqB.01}. Let ${\mat{X},\mat{Y}\in\matRd{mn}{mn}}$ be arbitrary matrices.
\begin{enumerate}[label={(\alph*)},ref={\thetheorem(\alph*)},leftmargin=*,widest=a]
    \item\label{thmB.02a}
        If ${\mat{L}\,\mat{X}=\mat{X}}$, that is ${X_{pq}=0}$ for ${p\not\in\cK}$, then $\mat{X}$ and $\matPhif{\mat{X}}$ have the same nonzero eigenvalues and, in particular, ${\specrad{\mat{X}}=\specrad{\matPhif{\mat{X}}}}$.
    \item\label{thmB.02b}
        The matrix function satisfies ${\specrad{\matPhif{\mat{X}}}\leq\norm{\matPhif{\mat{X}}}\leq\norm{\mat{X}}}$.
    \item\label{thmB.02c}
        If ${\mat{X}\,\mat{L}=\mat{X}}$ or ${\mat{L}\,\mat{Y}=\mat{Y}}$, then ${\matPhif{\mat{X}\,\mat{Y}}=\matPhif{\mat{X}}\matPhif{\mat{Y}}}$.
\end{enumerate}
\end{proposition}

\begin{proof}{Proposition~\ref{thmB.02}}
\begin{enumerate}[label={(\alph*)},leftmargin=*,widest=a]
    \item
        The rows of zeros, and the corresponding columns, can be permuted so that the resulting matrix has all zeros at the bottom. See \S0.9 and \S6.2 of \cite{HornJohnson2013} for further information. That is, we can write
        \begin{equation} \label{eqB.02}
            \mat{X} = \mat{M}^{-1} \twobytwomatrix{\matPhif{\mat{X}}}{\mat{B}}{\mat{0}}{\mat{0}} \mat{M},
        \end{equation}
        where ${\mat{M}\in\matRd{mn}{mn}}$ is an invertible permutation matrix and ${\mat{B}\in\matRdnn{\rb{m-1}n}{\rb{m-1}n}}$ is another matrix. The conclusion follows.
    \item
        The inequality ${\specrad{\matPhif{\mat{X}}}\leq\norm{\matPhif{\mat{X}}}}$ follows from \eqref{eqA.02}. Since $\matPhif{\mat{X}}$ is formed by expunging certain rows and columns of $\mat{X}$, the inequality ${\norm{\matPhif{\mat{X}}}\leq\norm{\mat{X}}}$ follows from the fact that the norm is computed as the maximum absolute column sum.
    \item
        From the assumptions, ${\mat{X}\,\mat{Y}=\mat{X}\,\mat{L}\,\mat{Y}}$ and hence ${\matPhif{\mat{X}\,\mat{Y}}=\mat{G}\,\mat{X}\,\mat{L}\,\mat{Y}\,\mat{H}}$. We know from Proposition~\ref{thmB.01a} that ${\mat{L}=\mat{H}\,\mat{G}}$, and hence we have ${\matPhif{\mat{X}\,\mat{Y}}=\rb{\mat{G}\,\mat{X}\,\mat{H}}\rb{\mat{G}\mat{Y}\,\mat{H}}}$. The conclusion follows.
\end{enumerate}
\end{proof}

\begin{proposition} \label{thmB.03}
Consider the matrices $\mat{L}$, $\mat{G}$, and $\mat{H}$, given in Table~\ref{tabB.01}, and the matrix function $\matPhi$, given in \eqref{eqB.01}. Suppose ${\mat{X}\in\matRd{mn}{mn}}$ is arbitrary and take ${\bmat{X}:=\matPhif{\mat{X}}}$.
\begin{enumerate}[label={(\alph*)},ref={\thetheorem(\alph*)},leftmargin=*,widest=a]
    \item\label{thmB.03a}
        If $\lambda$ is an eigenvalue of $\bmat{X}$ with associated right eigenvector $\vec{v}$, then $\lambda$ is an eigenvalue of ${\mat{L}\,\mat{X}}$ with associated eigenvector ${\vec{u}:=\mat{H}\,\vec{v}}$ (which satisfies ${\mat{L}\,\vec{u}=\vec{u}}$).
    \item\label{thmB.03b}
        If ${\mat{L}\,\mat{X}=\mat{X}}$ and $\lambda$ is an eigenvalue of $\mat{X}$ with associated left eigenvector $\vec{u}$, then $\lambda$ is an eigenvalue of $\bmat{X}$ with associated left eigenvector ${\vec{v}:=\mat{G}\,\vec{u}}$.
\end{enumerate}
\end{proposition}

\begin{proof}{Proposition~\ref{thmB.03}}
\begin{enumerate}[label={(\alph*)},leftmargin=*,widest=a]
    \item
        We are given ${\bmat{X}\,\vec{v}=\lambda\,\vec{v}}$. Since ${\bmat{X}=\mat{G}\,\mat{X}\,\mat{H}}$, left-multiplication by $\mat{H}$ yields ${\rb{\mat{H}\,\mat{G}}\mat{X}\rb{\mat{H}\,\vec{v}}=\lambda\rb{\mat{H}\,\vec{v}}}$. Since ${\mat{H}\,\mat{G}=\mat{L}}$ and ${\vec{u}=\mat{H}\,\vec{v}}$, we are finished.
    \item
        We are given that ${\vec{u}^*\rb{\mat{L}\,\mat{X}}=\lambda\,\vec{u}^*}$. Multiply on the right by $\mat{H}$ to obtain ${\vec{u}^*\rb{\mat{L}\,\mat{X}\,\mat{H}}=\lambda\rb{\vec{u}^*\,\mat{H}}}$. Substituting ${\mat{L}=\mat{H}\,\mat{G}}$ reveals ${\rb{\vec{u}^*\,\mat{H}}\bmat{X}=\lambda\rb{\vec{u}^*\,\mat{H}}}$. Since ${\mat{G}\,\vec{u}=\vec{v}}$ and ${\mat{H}^*=\mat{G}}$, we know ${\vec{v}^*=\vec{u}^*\,\mat{H}}$ and we are left with ${\vec{v}^*\,\bmat{X}=\lambda\,\vec{v}^*}$, as desired.
\end{enumerate}
\end{proof}


\paragraph{Main proofs.} With the preceding preparations completed, we now present the proofs for results in the main text that have not been omitted for brevity.

\begin{proof}{Proposition~\ref{thm02.04}}
\begin{enumerate}[label={(\alph*)},leftmargin=*,widest=a]
    \item
        It follows from \eqref{eqA.01}, \eqref{eq02.01}, and \eqref{eq02.03}.
    \item
        The proof is standard. The first statement, ${\sum_{p=1}^{mn}D_{pq}=1}$ for every ${q\in\cE_{mn}}$, follows from \eqref{eq02.02}, \eqref{eq02.04}, and \eqref{eq02.05}. To show that ${\norm{\mat{D}\,\vec{x}}=\norm{\vec{x}}}$ and ${\norm{\mat{D}\,\mat{X}}=\norm{\mat{X}}}$, employ routine algebra and use \eqref{eqA.01} and the column stochasticity of $\mat{D}$. To show that ${\specrad{\mat{D}}=\norm{\mat{D}}=1}$, first observe that from \eqref{eqA.01}, \eqref{eq02.02}, \eqref{eq02.04}, and \eqref{eq02.05}, we know ${\norm{\mat{D}}=1}$. Appealing to \eqref{eqA.02}, ${\specrad{\mat{D}} \leq 1}$. To show ${\specrad{\mat{D}}=1}$, show that the vector ${\vec{u}:=\rb{1,\ldots,1}\in\Rd{mn}}$ is a left eigenvector associated with the eigenvalue $1$. It follows ${\specrad{\mat{D}}=1}$.
\end{enumerate}
\end{proof}

\begin{proof}{Proposition~\ref{thm02.08}}
The proof is similar to that of Proposition~\ref{thm02.04}.
\end{proof}

\begin{proof}{Proposition~\ref{thm02.09}}
It follows from \eqref{eq02.13} and Proposition~\ref{thmB.01} that ${\hcR_0=\maxst{\norm{\bmat{W}\,\vec{y}}}{\vec{y}\in\cY}}$, where $\cY$ is the set given in \eqref{eq02.07}. If the $j^\text{th}$ column of $\bmat{W}$ gives the maximum absolute column sum, then choosing $\vec{y}$ to be the $j^\text{th}$ standard unit basis vector reveals the desired conclusion.
\end{proof}

\begin{proof}{Theorem~\ref{thm02.11}}
We know that ${\mat{N}=\mat{D}\,\mat{W}}$. Using Propositions~\ref{thmB.01b} and \ref{thmB.02c}, we can conclude ${\bmat{N}=\bmat{D}\cdot\bmat{W}}$. The facts ${\specrad{\bmat{N}}=\specrad{\mat{N}}=\cR_0}$ and ${\specrad{\bmat{W}}=\specrad{\mat{W}}}$ follow from Propositions~\ref{thmB.01b} and \ref{thmB.02a}. The facts  ${\cR_0\leq\norm{\bmat{N}}\leq\norm{\mat{N}}}$ and ${\hcR_0=\norm{\bmat{W}}\leq\norm{\mat{W}}}$ follow from Proposition~\ref{thmB.02b}.
\end{proof}

\begin{proof}{Corollary~\ref{thm02.12}}
It follows from \eqref{eq02.02}, Proposition~\ref{thm02.08}, and Theorem~\ref{thm02.11}.
\end{proof}

\begin{proof}{Theorem~\ref{thm02.13}}
\begin{enumerate}[label={(\alph*)},leftmargin=*,widest=a]
    \item
        By virtue of Theorem~\ref{thm02.11}, ${\cR_0\leq\norm{\bmat{N}}}$ and ${\hcR_0=\norm{\bmat{W}}}$. Using Proposition~\ref{thm02.08}, ${\norm{\bmat{N}}=\norm{\bmat{D}\cdot\bmat{W}}=\norm{\bmat{W}}}$. Thus, ${\cR_0\leq\hcR_0}$.
    \item
        Define the function ${\vec{g}\in\Cr{}{\cY,\Rdnn{}}}$ by ${\vecf{g}{\vec{y}}:=\norm{\bmat{N}\,\vec{y}}=\sum_{j=1}^n\rb{\sum_{i=1}^n\bN_{ij}}y_j}$. Note that $\cY$ is nonempty, compact, and connected. Let $j_\alpha$ and $j_\beta$ be such that ${\sum_{i=1}^{n}\bN_{ij_\alpha}=\alpha}$ and ${\sum_{i=1}^{n}\bN_{ij_\beta}=\beta}$. Define the vectors ${\vec{a},\vec{b}\in\cY}$ component-wise using the Kronecker-delta by ${a_i:=\delta_{ij_\alpha}}$ and ${b_i:=\delta_{ij_\beta}}$ for each ${i\in\cE_n}$. Then, ${\alpha=\norm{\bmat{N}\,\vec{a}}=\vecf{g}{\vec{a}}}$ and ${\beta=\norm{\bmat{N}\,\vec{b}}=\vecf{g}{\vec{b}}}$ with ${\vecf{g}{\vec{a}}\leq\vecf{g}{\vec{y}}\leq\vecf{g}{\vec{b}}}$ for each ${\vec{y}\in\cY}$.

        Since ${\cR_0=\specrad{\bmat{N}}}$, as we know from Theorem~\ref{thm02.11}, it follows from \eqref{eqA.02} that ${\alpha\leq\cR_0\leq\beta}$. With the assistance of the Intermediate Value Theorem, we know there exists ${\veczeta\in\cY}$ such that ${\cR_0=\vecf{g}{\veczeta}=\norm{\bmat{N}\,\veczeta}}$. Using Propositions~\ref{thmB.01b}, \ref{thmB.01d}, and \ref{thmB.01e}, ${\cR_0=\norm{\mat{N}\,\vecxi}}$, where ${\vecxi:=\mat{H}\,\veczeta\in\cX}$.
    \item
        The Perron-Frobenius Theorem guarantees the existence of ${\veczeta\in\Rnp}$ with ${\norm{\veczeta}=1}$ and ${\bmat{N}\,\veczeta=\cR_0\,\veczeta}$. Observe that ${\veczeta\in\cY}$ and so, by virtue of Proposition~\ref{thmB.01d}, ${\vecxi\in\cX}$, where ${\vecxi:=\mat{H}\,\veczeta}$. Using Proposition~\ref{thmB.01b} and Proposition~\ref{thmB.03a}, ${\mat{N}\,\vecxi=\cR_0\,\vecxi}$. Taking the norm of both sides of this equation, ${\norm{\mat{N}\,\vecxi}=\cR_0\norm{\vecxi}=\cR_0}$, as requested.
    \item
        The proof is straight-forward and omitted.
\end{enumerate}
\end{proof}

\begin{proof}{Proposition~\ref{thm03.01}}
\begin{enumerate}[label={(\alph*)},leftmargin=*,widest=a]
    \item
        Since ${\mat{D}\,\mat{S}=\mat{S}}$ and ${\mat{W}=\mat{F}\rb{\mat{I}-\mat{D}\,\mat{S}}^{-1}}$, it is easy to see from the block-diagonal forms of $\mat{F}$ and $\mat{S}$ that ${\mat{W}=\bigoplus_{i=1}^n\matm{F}{i}\sqb{\mat{I}-\matm{S}{i}}^{-1}=\bigoplus_{i=1}^n\matm{N}{i}}$. Upon expunging the rows and columns for non-newborns and recalling that ${\scim{N}{11}{i}=\scim{\cR}{0}{i}}$, we observe ${\bmat{W}=\bigoplus_{i=1}^n\scim{\cR}{0}{i}}$.
    \item
        Since ${\hcR_0=\norm{\bmat{W}}}$, it follows from \eqref{eqA.01} that ${\hcR_0=\max{i\in\cE_n}{\scim{\cR}{0}{i}}}$. Appealing to Proposition~\ref{thm02.08} (specifically ${\bD_{ij}=\scim{d}{1}{i,j}}$) and Theorem~\ref{thm02.11} (specifically ${\bmat{N}=\bmat{D}\cdot\bmat{W}}$), ${\bN_{ij}=\scim{d}{1}{i,j}\,\scim{\cR}{0}{j}}$ for each ${i,j\in\cE_n}$.
    \item
        Now, ${\alpha\,\bmat{D}\leq\bmat{N}\leq\beta\,\bmat{D}}$ and ${\specrad{\bmat{D}}=1}$, the latter of which we know from Proposition~\ref{thm02.08}. Since ${\cR_0=\specrad{\bmat{N}}}$, the Perron-Frobenius Theorem (specifically, increasing an element of a positive matrix increases the spectral radius) establishes the conclusions.
\end{enumerate}
\end{proof}

\begin{proof}{Proposition~\ref{thm03.06}}
First, note that $0$ cannot be a root of $g$ since ${\scf{g}{0}=-b<0}$. To prove the first two statements, sketch the graphs of ${y=u^2}$ and ${y=a+\fracslash{b}{u}}$ to illustrate that there is a unique positive root $u^*$ which satisfies ${\pderivslash{u^*}{a}>0}$ and ${\pderivslash{u^*}{b}>0}$ and is greater in magnitude than any other root (which must be negative). Alternatively, use Descartes' Rule of Signs to show that $g$ has a single positive root $u^*$ and, after employing routine Calculus to show that ${3\rb{u^*}^2-a>0}$, use implicit differentiation to show ${\pderivslash{u^*}{a}=\fracslash{u^*}{\sqb{3\rb{u^*}^2-a}}>0}$ and ${\pderivslash{u^*}{b}=\fracslash{1}{\sqb{3\rb{u^*}^2-a}}>0}$.

To prove the third statement, observe that ${\scf{g}{1}=1-\rb{a+b}}$ and note \dsm{\lim_{u\to\infty}\scf{g}{u}=\infty}. So, if ${a+b=1}$ then ${\scf{g}{1}=0}$ and ${u^*=1}$. Likewise, if ${a+b<1}$ then ${\scf{g}{1}>0}$ and so ${u^*<1}$ and if ${a+b>1}$ then ${\scf{g}{1}<0}$ and so ${u^*>1}$. That is, $u^*$ and ${a+b}$ are on the same side of 1.
\end{proof}

\section*{Acknowledgements}

\addcontentsline{toc}{section}{Acknowledgements}

We would like to thank Farah Abu Sharkh, Dr.~Zou's summer student in the summer of 2011, for her help. We would also like to thank, for their generous financial support, the Natural Sciences and Engineering Research Council (NSERC) of Canada (X.Z.), the Mathematics of Information Technology and Complex Systems (MITACS) Elevate Strategic Fellowship Program (M.S.C.), MITACS Networks of Centres of Excellence (NCE) (X.Z.), and the Ontario Funding for Canada-Ontario Agreement (COA-7-22) Respecting the Great Lakes Basin Ecosystem (M.S.C., Y.Z.).

\end{document}